\documentclass[11pt,reqno,letterpaper]{amsart}
\usepackage{amssymb}
\usepackage{graphicx,color}

\usepackage[normalem]{ulem}
\usepackage{fullpage} 

\usepackage{color}
\usepackage{amsxtra}
\usepackage{amssymb}
\usepackage{mathtools}
\usepackage{enumerate}
\usepackage{nicefrac}
\mathtoolsset{showonlyrefs} 
\usepackage{ mathrsfs }
\usepackage{tikz}
\usetikzlibrary{positioning, calc}
\usetikzlibrary{shadings}
\usepackage{color}
\usepackage{hyperref}
\hypersetup{
  colorlinks   = true, 
  urlcolor     = blue, 
  linkcolor    = blue, 
  citecolor   = red 
}
  
\newtheorem{theorem}{Theorem}[section]
\newtheorem{lemma}{Lemma}[section]
\newtheorem{proposition}{Proposition}[section]

\theoremstyle{definition}
\newtheorem{definition}{Definition}[section]
\newtheorem{remark}{Remark}[section]
\newtheorem{example}{Example}[section]

\newcommand{\R}{{\mathbb R}}

\newcommand{\RR}{{\mathbb R}}

\newcommand{\N}{{\mathbb N}}

 \def\1{\raisebox{2pt}{\rm{$\chi$}}}

\usepackage{enumerate}

\def\undertilde#1{\mathord{\vtop{\ialign{##\crcr
$\hfil\displaystyle{#1}\hfil$\crcr\noalign{\kern1.5pt\nointerlineskip}
$\hfil\tilde{}\hfil$\crcr\noalign{\kern1.5pt}}}}}

\DeclareMathAlphabet{\mathpzc}{OT1}{pzc}{m}{it}
\newcommand{\dist}{{\mathpzc{d}}}
\title{
Fractional convexity
}

\author[L. M. Del Pezzo, A. Quaas and J.D. Rossi]{Leandro M. Del Pezzo$^1$, Alexander Quaas$^2$ and Julio D. Rossi$^1$}

\begin{document}

\maketitle

         \centerline{$^1$Departamento  de Matem{\'a}tica, FCEyN,}
         \centerline{Universidad de Buenos Aires, }
         \centerline{Ciudad Universitaria, Pabellon I, (C1428BCW),}
        \centerline{Buenos Aires, Argentina. }
\centerline{ldpezzo@dm.uba.ar, jrossi@dm.uba.ar}

\bigskip

        \centerline{$^2$Departamento de Matem\'atica,} 
       \centerline{Universidad T\'ecnica Federico Santa Mar\'ia }
        \centerline{Casilla V-110, Avda. Espa\~na, 1680}
        \centerline{Valpara\'iso, Chile.}
    \centerline{alexander.quaas@usm.cl}

\begin{abstract} 
	We introduce a notion of fractional convexity that extends naturally 
	the usual notion of convexity in the Euclidean space to a fractional setting. 
	With this notion of fractional convexity, we study the fractional convex envelope inside 
	a domain of an exterior datum (the largest possible fractional convex function 
	inside the domain that is below the datum outside)
 	and show that the fractional convex envelope is characterized as a
	viscosity solution to a non-local equation that is given by the 
	infimum among all possible
	directions of the $1-$dimensional fractional laplacian. For this equation we prove
	existence, uniqueness and a comparison principle (in the framework of viscosity solutions).
	In addition, we find that solutions
	to the equation for the convex envelope are related to solutions to the fractional Monge-Ampere equation. 
\end{abstract}

\section{Introduction}
	The purpose of this paper is to provide a notion of convexity
	in the fractional setting.

	First, let us recall the usual notion of convexity in the Euclidean space. 
	We fix a bounded smooth domain $\Omega \subset {\mathbb{R}}^N$.
	A function $u\colon \Omega \to  {\mathbb{R}}$ is said to be convex in $\Omega$ if,
	for any two points $x,y \in \Omega$ such that the segment $[x,y]:=\{tx+(1-t)y: t \in (0,1)\}$ (the line segment
	connecting $x$ and $y$) is contained in $\Omega$, it holds that
	\begin{equation} \label{convexo-usual}
		u(tx+(1-t)y) \leq tu(x) + (1-t) u(y), \qquad \forall t \in (0,1).
	\end{equation}
	Notice that $t \mapsto v(tx+(1-t)y) := tu(x) + (1-t) u(y)$ is just the solution to the equation 
	$v''=0$ in the segment $[x,y]$ that verifies $v(x) = u(x) $ and $v(y) = u(y)$ at the endpoints.
		We refer to \cite{Vel} for a general reference on convexity. 
	
	With this notion of convexity one can define the convex envelope inside $\Omega$ of a boundary
	datum $g\colon \partial \Omega \to  {\mathbb{R}}$ as
	\begin{equation} \label{convex-envelope-usual}
		u^* (x) \coloneqq \sup \Big\{v(x) \colon v 
		\mbox{ is convex in $\overline{\Omega}$ and verifies } v|_{\partial \Omega} \leq g \Big\}.
	\end{equation}
	
	In terms of a second order partial differential equation (PDE), a function is convex if and only if 
	\[
		\lambda_1 (D^2 u) (x)\coloneqq
		\inf \Big\{
			\langle D^2 u(x) z, z \rangle\colon z\in\mathbb{S}^{N-1}
		\Big\} \geq 0
	\] 
	in the viscosity sense. 
	Here $\mathbb{S}^{N-1}$ denotes the $(N-1)-$sphere, that is 
	$\mathbb{S}^{N-1}\coloneqq \{z\in \mathbb{R}^N \colon |z|=1\}$. 
	
		Moreover,
	the convex envelope of $g$, a continuous datum on the boundary, in a strictly convex domain 
	turns out to be the unique solution to
	\begin{align}
		\label{convex-envelope-usual-eq-1}	\lambda_1 (D^2 u) (x) = 0 \qquad &x\in  \Omega, \\[6pt]
		\label{convex-envelope-usual-eq-2}	u(x) = g (x)  \qquad &x\in  \partial \Omega.
	\end{align}
	
		The equation \eqref{convex-envelope-usual-eq-1} has to be
	interpreted in viscosity sense and the boundary condition \eqref{convex-envelope-usual-eq-2} is attained with continuity. 
	We refer to \cite{BlancRossi,HL1,OS,Ober}, and references therein.
	
	Notice that $\lambda_1$ is the smallest eigenvalue of the Hessian, that is, if
	$\lambda_1\leq \lambda_2\leq\cdots\leq\lambda_N$ are the ordered eigenvalues of
	the Hessian matrix, $D^2u$, then the equation reads as $\lambda_1=0$. 
	Also remark that 
	\[
	\lambda_1 (D^2 u) (x)=\inf \Big\{\langle D^2 u(x) z, z \rangle\colon z\in\mathbb{S}^{N-1}\Big\},
	\] 
	says that the operator that is associated to the convex envelope
	is just the infimum of the second directional derivatives of the function among all possible directions.
	
	\medskip
	
	Now we propose the following natural extension of convexity to the fractional setting. 
	Given $s\in(0,1),$
	a function $u\colon\mathbb{R}^N \to  {\mathbb{R}}$ is said to be {\it $s-$convex} in $\Omega$ if
	for any two points $x,y \in \Omega$ such that the segment $[x,y]$ is contained in 
	$\Omega$ it holds that
	\begin{equation} \label{convexo-s}
		u(t x+(1-t)y) \leq v(t x+(1-t)y), \qquad \forall t \in (0,1)
	\end{equation}
	where $v$ is just the viscosity solution 
	to $\Delta^s_1 v =0$ (the 1-dimensional $s-$fractional laplacian)
	in the segment $[x,y]$ with $v= u$ outside the segment. That is, $v$ verifies
	\[
		\Delta^s_1 v(t x+(1-t )y)
		\coloneqq C(1,s)\int_{\mathbb{R}} \frac{v(rx+(1-r)y)-v(t x+(1-t)y)}{|r-t|^{1+2s}} \, dr = 0
	\]
	for every $t \in (0,1)$ with 
	\[
		v(z)=u(z) \qquad \mbox{ for } z=tx+(1-t)y \mbox{ with }t\not\in (0,1)
	\] 
	(as usual for the fractional laplacian we have to impose an exterior datum). 
	Here $C(1, s)$ is a normalization constant whose value is irrelevant 
	for our arguments 
	(and hence it will be omitted from now on) and the integral is to be understood 
	in the principal value 
	sense. Notice that we have to use values of $u$ outside $\Omega$ since 
	the involved operator is nonlocal, therefore
	$u$ has to be defined in the whole $\mathbb{R}^N$. In addition, 
	we need that such a function $v$ is well defined, and then we have to impose 
	some additional conditions on $u,$ that is, $u$ is a locally bounded function and 
	\[
		t\mapsto u(x+tz)\in 
		L_{s}(\mathbb{R})\coloneqq\left\{f\in L^1_{loc}(\mathbb{R})\colon 
		\int_{\mathbb{R}^N}\dfrac{|f(r)|}{(1+|r|)^{1+2s}}dr <\infty\right\}
	\]
	for any $x\in\Omega$ and any $z\in\mathbb{S}^{N-1}.$
	The space
		$L_s(\mathbb{R})$ is the right space for which 
		\[
			\mu(x)=\int_{|x-y|>\varepsilon} \frac{u(y)-u(y)}{|x-y|^{1+2s}} dy
		\] 
	exists for every $\varepsilon>0.$ Moreover $\mu$ is continuous
	at the continuity points of $u.$ See for instance \cite{Garofolo, kkl, Biccari}.

	With this definition of $s-$convexity one can define the {\it $s-$convex envelope} of an exterior 
	datum $g \colon \mathbb{R}^N\setminus \Omega \to  {\mathbb{R}}$ as
	\begin{equation} \label{convex-envelope-s}
		u^* (x) \coloneqq \sup \left\{w(x) \colon w \mbox{ is $s-$convex in $\overline{\Omega}$ and verifies } 
			w|_{\mathbb{R}^N\setminus \Omega} \leq g \right\}.
		\end{equation}
	This definition makes sense when the above set of functions is not empty (in particular, this is the case
	when there exists an extension of $g$ inside $\Omega$ that is $s-$convex and from our results this 
	holds when
	$g$ is continuous and bounded). The function $u^* (x)$ is unique and $s-$convex 
	(it follows from the comparison
	principle for the fractional $s-$laplacian in $1-$dimension that the supremum 
	of $s-$convex functions is also $s-$convex).
	
	\medskip
	
	Our main result is the following:
	
	\begin{theorem} \label{teo.1.intro} 
		Assume that $\Omega$ is a {bounded} strictly convex $C^2-$domain,
		and that $g$ is continuous and bounded. 
		Then, the $s-$convex envelope is well defined and is continuous 
		in $\overline{\Omega}$ (up to the boundary) with 
		$u|_{\partial \Omega} =g|_{\partial \Omega}$ (therefore the exterior datum is taken with continuity). 
		
		Moreover, the $s-$convex envelope is characterized as being the
		unique viscosity solution to
		\begin{equation} \label{convex-envelope-s-eq} \left\{
			\begin{array}{ll}
				\displaystyle \Lambda_1^s u (x) \coloneqq \inf 
				\left\{ \int_{\mathbb{R}} 
					\frac{u(x+tz)-u(x)}{|t|^{1+2s}} 
					\, dt\colon z\in \mathbb{S}^{N-1}
				\right\} = 0 \quad & x \in \Omega ,\\[6pt]
				u (x) = g (x) \quad & x \in \mathbb{R}^N\setminus \Omega.
			\end{array}
			\right.
		\end{equation}
	\end{theorem}
	
	In the course of the proof of our main result, we also obtain the following characterization of being 
		$s-$convex: a function $u\colon \mathbb{R}^N \to \mathbb{R}$ is $s-$convex in $\Omega$ if and only if 
		\begin{equation} \label{convex-s-eq}
				\begin{array}{ll}
					\displaystyle \Lambda_1^s u (x)  
					\geq  0 \qquad & x \in \Omega
				\end{array}
			\end{equation}
			in the viscosity sense.
			
		\subsection*{Classical convexity vs. fractional convexity}
		We also compare our notion of fractional convexity with the usual convexity obtaining that for 
			$s>1/2$ a classical convex function in the whole space $\mathbb{R}^N$ is $s-$convex, for details see Proposition \ref{pop:convex_implies_sconvex} below; while in a bounded domain we present simple examples showing that the usual convexity
			and the fractional convexity are different notions (none implies the other). 
	
	\subsection*{The first fractional eigenvalue}
	Remark that for the $s-$convex envelope we have an integral equation that is given by 
	the infimum among all possible directions of the $1-$dimensional fractional $s-$laplacian computed 
	at the point $x$. We call this fractional operator that is associated with this notion of 
	fractional convexity $\Lambda_1^s u$ in analogy with the first eigenvalue of the Hessian, 
	$\lambda_1 (D^2 u)$, that is given by the infimum among all directions of the $1-$dimensional 
	second derivative and is associated with the classical notion of convexity. 
	Hence, we think $\Lambda_1^{s} u$ as the \lq\lq first fractional eigenvalue".
	
	\subsection*{On our hypotheses on the data, $\Omega$ and $g$} Notice that the hypothesis that 
	$\Omega$ is strictly convex
	is used in order to show that the $s-$convex
		 is continuous up to the boundary for an exterior datum $g$ continuous and bounded. For the 
	classical notion of convexity this geometric
	condition also appears naturally and is necessary and also sufficient to obtain that a continuous 
	boundary datum is attained continuously 
	(the convex envelope of the datum is continuous in $\overline{\Omega}$), 
	see \cite{BlancRossi,OS}.

		Remark that for our definition to make sense 
		we need to assume that the exterior datum $g$ is such that we can solve the Dirichlet problem
		for the $1-$dimensional fractional $s-$laplacian in every segment inside $\Omega$ 
		(this involves values of $g$ in the line that contains this segment). 
		We ask that the datum $g$ is continuous and bounded 
		(and this guarantees that
		there is a solution for the $1-$dimensional fractional $s-$laplacian in every segment 
		inside $\Omega$ with exterior datum $g$ that is 
		uniformly bounded by a bound for $|g|$).  
		However, slightly more general data can be also considered 
		(as long as we have solvability and equiboundedness
		 of all these $1-$dimensional problems, 
		 notice that 
		 $t\to g(x+tz)\in L_s(\mathbb{R})$ is enough).

	\subsection*{Localization of $s-$convexity} 
	One can localize $s-$convexity in $\Omega$ and use only values of $u$ inside the domain just
	computing the $1-$dimensional fractional operator restricting the domain of integration to 
	the intersection of the line with $\Omega$ 
	(thus we avoid the need to consider values of $u$ outside $\Omega$). We will briefly comment
	on this localization in Section \ref{sect-localiz}.
	
	\subsection*{$s-$concavity} \label{rem.concavas} 
		As for the local case, we will say that a function $u$ is {\it $s-$concave} 
		if $-u$ is $s-$convex. Similar results can be proved for the $s-$concave envelope defined as 
		\[
			u_* (x) \coloneqq \inf 
			\left
				\{v(x) \colon v \mbox{ is $s-$concave in $\overline{\Omega}$ 
				and verifies } v|_{\mathbb{R}^N\setminus \Omega} 
				\geq g 
			\right\}.
		\]
		In this case the equation that appears is 
		\[
			\Lambda_N^s (u) (x) \coloneqq \sup
				\left\{ \int_{\mathbb{R}} \frac{u(x+tz)-u(x)}{|t|^{1+2s}} \, dt \colon z\in\mathbb{S}^{N-1}
			\right\}=0
		\]
		that is analogous to the largest eigenvalue of $D^2u$ 
		\[
			\lambda_N (D^2 u) (x)\coloneqq
			\sup \Big\{\langle D^2 u(x) z, z \rangle\colon z\in\mathbb{S}^{N-1}\Big\}=0,
		\] 
		that holds for the classical notion of concave envelope, \cite{BlancRossi,OS}.

	\subsection*{Relation with a nonlocal Monge-Ampere equation}
		Solutions to 
		\[
			\Lambda_1^{s} u (x) = 0
		\]
		are also solutions to the nonlocal version of Monge-Ampere introduced in \cite{Charro-Caffa},
		\[
	 		\inf_{A\in L} \int_{\mathbb{R}^N} \frac{u(y)-u(x)}{|A^{-1} (y-x)|^{N+2s} }dy =0, 
	 	\]
	 	where $L$ corresponds to the family of symmetric positive matrices with determinant 1,
	 	\[
	 		L\coloneqq\left\{  A\in \mathbb{S}^{N\times N} \colon A>0, \, \det(A)=1  \right\}.
	 	\]
		In fact, we have
		\[
			 \inf_{A\in L} \int_{\mathbb{R}^N} \frac{u(y)-u(x)}{|A^{-1} (y-x)|^{N+2s} }dy = 
	  		\inf_{A\in L} C \int_{\mathbb{R}} \int_{|z|=1} 
	  		\frac{u(x+tz)-u(x)}{|t|^{1+2s} |A^{-1} z|^{N+2s} } dz dt. 
	 	\]
	 	Hence, if $\Lambda_1^{s} u (x) \geq 0$ we get
	 	\[
	 		\inf_{A\in L} \int_{\mathbb{R}^N} \frac{u(y)-u(x)}{|A^{-1} (y-x)|^{N+2s} }dy =
	  		\inf_{A\in L} C \int_{|z|=1} \frac{1}{ |A^{-1} z|^{N+2s}} \int_{\mathbb{R}}  
	  		\frac{u(x+tz)-u(x)}{|t|^{1+2s} } dt dz \geq 0.
	 	\]
	 	
	 	On the other hand, if $\Lambda_1^{s} u (x) \leq 0$, there exists a sequence of directions 
	 	$z_n$, $|z_n|=1$ with 
		\[
			 \lim_{n\to\infty} 
			 \int_{\mathbb{R}}  \frac{u(x+tz_n)-u(x)}{|t|^{1+2s} } dt \leq 0.
	 	\]
	 Now, one can take a sequence of matrices $A_n\in L $ 
	 with an eigenvalue of order $n$ in the direction of $z_n$ (and 
	 all the other eigenvalues go to $0$ as $n\to\infty$) 
	 to obtain 
	 \begin{align*}
	 	\displaystyle 
	 	\inf_{A\in L} \int_{\mathbb{R}^N} \frac{u(y)-u(x)}{|A^{-1} (y-x)|^{N+2s} }dy &=
	  	\inf_{A\in L} C \int_{|z|=1} \frac{1}{ |A^{-1} z|^{N+2s}} \int_{\mathbb{R}}  
	  		\frac{u(x+tz)-u(x)}{|t|^{1+2s} } dt dz \\[6pt]
	  	&\displaystyle \leq \lim_{n\to\infty}  \int_{|z|=1} \frac{1}{ |A_n^{-1} z|^{N+2s}} 
	  	\int_{\mathbb{R}}  \frac{u(x+tz)-u(x)}{|t|^{1+2s} } dt dz\\[6pt] 
	  	&\leq 0.
	 \end{align*}
	 Notice that all these computations can be justified in a viscosity sense.
	 
	This is analogous to what happens in the local case, where solutions to $\lambda_1 (D^2u) =0$
	are also convex solutions to the local Monge-Ampere equation $\det(D^2u)=0$, see \cite{OS}.
	
	This relation with this nonlocal version of Monge-Ampere reinforces 
	the intuitive idea that $\Lambda_1^{s} u $ is the \lq\lq first fractional eigenvalue".
	
	Monge-Ampere equations and convex envelopes of a given function and 
	their contact sets play a crucial role when proving 
	Aleksandrov-Bakelman-Pucci (ABP) estimates for elliptic differential equations. 
	We would like to remark that, 
	for ABP in the fractional case we refer to \cite{paper} 
	where an envelope of a given function is defined using ideas similar to ours.

	\subsection*{Notations} 
		Throughout this paper $\Omega\subset\mathbb{R}^N$ is a bounded strictly convex  $C^2-$domain. 
		Given $x\in\Omega$ and $z\in\mathbb{S}^{N-1},$ we will denote by $L_{z}(x)$ 
		the line 
		that passes through $x$ and has direction $z$, that is,  
		\[
			L_{z}(x)\coloneqq\{x+tz\colon t\in\mathbb{R}\}.
		\]
		For any function $u,$ the positive and negative parts of $u$ are denoted by
		\[
			u_+(x)\coloneqq\max\{u(x),0\} \quad \text{ and }\quad  u_{-}(x)\coloneqq\max\{-u(x),0\}.
		\]
	
		Finally, we assume that the signed distance function to $\partial\Omega$ is positive in $\Omega$ 
		and negative in $\RR^N\setminus \overline{\Omega}.$  Throughout the rest of this article,
		$\dist$ denotes a $C^2-$function in $\RR^N$ which agrees with the signed distance function to 
		$\partial\Omega$ in a neighbourhood of $\partial \Omega.$ 
	
	\subsection*{On the definition of being a viscosity solution to $\Lambda_1^s (u)=0$} 		
		\label{sobre.def.sol.viscosa}
		Now we have to discuss the delicate issue of the notion of what is a viscosity solution in our context. 
	
		We notice that we have two notions of viscosity solution to 
		\[
			\begin{cases}
				\Lambda_1^s (u)=0 &\mbox{ in } \Omega,\\[6pt]
				u = g &\mbox{ in } \mathbb{R}^N\setminus \Omega,
			\end{cases}
		\]
		see Definitions \ref{defsol.44} and \ref{defi.2.5}. 

		In the first notion (that corresponds to what is usual in the viscosity theory) we test with $N-$dimensional functions 
		$\phi:\Omega \mapsto \mathbb{R}$ that touches
		$u$ from above at $x_0\in \Omega$ and
		we ask for 
		$$
		 \int_{\mathbb{R}} \frac{\tilde{\phi}(x_0+tz)-\phi(x_0)}{|t|^{1+2s}} \, dt \geq 0
		$$
		for any direction $z \in\mathbb{S}^{N-1}$. In computing the nonlocal $1-$dimensional operator we have
		taken $\tilde{\phi}$ such that $\tilde{\phi} =g$ outside $\Omega$ and $\tilde{\phi}=\phi$ near $x_0$. 
		We also assume the reverse inequality when the test function touches $u$ from below at $x_0$ in the 
		$N-$dimensional set $\Omega$. See Definition \ref{defsol.44}.

		The alternative definition (see Definition \ref{defi.2.5}) runs as follows: we take a direction 
		$z \in\mathbb{S}^{N-1}$ and then
		a 
		$1-$dimensional test function 
		$\phi$ that touches $u$ from above at $x_0$ 
		in the $1-$dimensional set  
		$L_{z}(x_0) \cap \Omega$. 
		Notice that now $\phi$ needs only to be defined in the $1-$dimensional set and not in the whole 
		$\Omega$. Here
		we ask for the same inequality,
		$$
		 \int_{\mathbb{R}} \frac{\tilde{\phi}(x_0+tz)-\phi(x_0)}{|t|^{1+2s}} \, dt \geq 0,
		$$
		with $\tilde{\phi}(x_0+tz) = g(x_0+tz)$ for $x_0+tz\not\in \Omega$ and 
		$\tilde{\phi}(x_0+tz)= {\phi}(x_0+tz)$ for $t$
		near $0$.
		As before, we also assume the reverse inequality when the test function touches 
		$u$ from below at $x_0$ in 
		the whole $\mathbb{R}^N$.

		Observe that if $u$ is a viscosity solution according to this second definition then it is 
		a solution according to the first one.
		This is due to the fact that when an $N-$dimensional test function $\phi$ defined in $\Omega$ touches 	
		$u$ from above at $x_0$ in $\Omega$, then the restriction
		of $\phi$ to any segment in $L_{z}(x_0)$, $\phi (x_0+t z)$, touches $u$ from above at $x_0$ 
		in $L_{z}(x_0)\cap \Omega$.
		The converse also holds but is delicate since given a $1-$dimensional test function that touches $u$ 
		in a segment there is no immediate way of obtaining an $N-$dimensional test function $\psi$
		that touches $u$ in $\Omega$ and such that the restriction of $\psi$ to the segment is $\phi$ 
		(we need to extend $\phi$ smoothly from the segment to the whole $\Omega$ 
		and still be above or below $u$).

		\subsection*{Ideas used in the proofs} 
		Our strategy to prove Theorem \ref{teo.1.intro} and deal with the two notions of solution (that we will show here that are 
		equivalent) is the 
		following: First, we will show the existence and uniqueness of a viscosity solution in the sense of the 
		first definition (touching by test functions
		in the whole $\Omega$). 
		This is accomplished via Perron's method (proving the validity of a comparison principle).
		Next, we prove that the $s-$convex envelope of an exterior datum $g$ is a solution 
		to the PDE \eqref{convex-envelope-s-eq} 
		according to the first or to the second definition 
		(testing with $1-$dimensional functions on segments). 
		Therefore, from the previous discussion, it turns out that the $s-$convex envelope is a solution to 
		\eqref{convex-envelope-s-eq} testing both as usual in the whole $\Omega$ and with 
		$1-$dimensional functions and then, from the uniqueness of such solutions 
		(in the sense of the first definition),
		we conclude that the $s-$convex envelope is given by the unique
		viscosity solution to \eqref{convex-envelope-s-eq}. Besides, we
		obtained that the two notions of viscosity solution coincide. 
		In fact, a viscosity solution testing with $1-$dimensional functions on segments is
		a viscosity solution testing with $N-$dimensional tests and the unique viscosity solution 
		testing with $N-$dimensional tests
		coincides with the $s-$convex envelope that is a solution testing with $1-$dimensional functions.

\medskip

\subsection*{The paper is organized as follows} In Section \ref{sect.2} we prove 
existence and uniqueness for viscosity solutions to the Dirichlet problem for $\Lambda_1^s u=0$
(these are consequence of the validity of a comparison result), here we use test functions touching $u$
in $\Omega$ and follow ideas from \cite{BarChasImb}; in Section \ref{sect.3} we start the analysis of $s-$convexity and
we show that  being a viscosity solution to $\Lambda_1^s u\geq 0$ (testing with $1-$dimensional functions
in segments inside $\Omega$) is equivalent to being $s-$convex; in Section \ref{sec:convexvssconvex} we compare the usual convexity with
the fractional convexity; in Section \ref{sectTeo11} we prove our main result, 
Theorem \ref{teo.1.intro}, that says that the $s-$convex envelope is characterized as the unique solution to $\Lambda_1^s u=0$
found in Section \ref{sect.2}; finally, in Section \ref{sect-localiz} we present an alternative way of defining $s-$convexity
using only values of $u$ in $\Omega$.

\section{Existence, uniqueness and a comparison principle for $\Lambda_1^s$.} \label{sect.2}

	The main result in this section is to prove a comparison principle for the problem 
	\begin{equation} \label{convex-33} 
		\begin{cases}
			\displaystyle \Lambda_1^s u (x) =\inf \left\{ \int_{\mathbb{R}} 
						\frac{u(x+tz)-u(x)}{|t|^{1+2s}} 
						\, dt\colon z\in \mathbb{S}^{N-1}
					\right\} = f(x) & x \in \Omega, \\[6pt]
					u (x) = g (x) & x \in \mathbb{R}^N\setminus \Omega.
		\end{cases}
	\end{equation}
	To this end, we borrow ideas from \cite{BarChasImb}.

\subsection{Basic notations and definition of solution.} 
	We use the notion of viscosity solution from \cite{BarChasImb}, which is the nonlocal extension of 
	the classical theory, see \cite{CIL}. 
	
	To state the precise definition of solution, we need the following: 
	Given $g\colon\R^N\setminus\Omega \to \R,$  for a function $u \colon \overline{\Omega} \to \R$
	we define the upper $g$-extension of $u$ as
	\begin{equation*}
		u^g(x) \coloneqq 
			\left \{ 
				\begin{array}{ll} 
					u(x) \quad & \mbox{if} \ x \in \Omega, \\[6pt] 
					g(x) \quad & \mbox{if} \ x \in \R^N\setminus\overline{\Omega},\\[6pt]
					\max \{ u(x), g(x) \} \quad & \mbox{if} \ x \in \partial \Omega.
				\end{array} 
			\right . 
	\end{equation*}
	In the analogous way we define $u_g$, the lower $g$-extension of $u$, replacing $\max$ by $\min$. 
	
	An important fact, that can be easily 
	verified, is that for 
	any continuous function $g \colon \R^N\setminus\Omega \to \R$ and
	any upper semicontinuous function $u \colon \overline{\Omega} \to \R,$  it holds that
	\[
		u^g = \tilde w,\quad \mbox{ with} \quad w=u \mathbf{1}_{\overline{\Omega}} + g \mathbf{1}_{\R^N\setminus\overline{\Omega}}
		\text{ in }\R^N.
	\] 
	
	Here we are using the definition of the upper (lower) semicontinuous envelope 
	$\tilde{w}$ ($\undertilde{w}$) of $w,$ that is,
	\[
			\tilde{w}(x)\coloneqq\inf_{r>0}\sup\{w(y)\colon y\in B(x,r)\}\qquad
			(\undertilde{w}(x)=\sup_{r>0}\inf\{w(y)\colon y\in B(x,r)\}),
	\]  
and $\mathbf{1}_{A}$
	denotes the indicator function of a set $A$ in $\mathbb{R}^N.$

\medskip
	We now introduce a useful notation, for $\delta > 0$ we write
		\[	
		E_{z, \delta}(u, \phi, x)\coloneqq I^1_{z, \delta}(\phi, x)+I^2_{z, \delta}(u, x)-f(x)
	\] 
	with
	\[
		\begin{array}{l}
			\displaystyle 
				I^1_{z, \delta}(\phi, x)
				\coloneqq\int_{-\delta}^\delta \frac{\phi (x+t z)-\phi(x)}{|t|^{1+2s}}dt,\\[15pt]
			\displaystyle  
				I^2_{z, \delta}(u, x)\coloneqq\int_{\RR \setminus (-\delta,\delta)} 
					\frac{u^g( x+t z)-u( x)}{|t|^{1+2s}}dt,
		\end{array}
	\]
	and then define 
	\begin{equation*}
		E_\delta(u, \phi, x) \coloneqq 
		- \inf\Big\{
			E_{z, \delta}(u, \phi, x)\colon 
				z\in \mathbb{S}^{N-1} 
			\Big\}.
	\end{equation*}

	\medskip
	
	Now we can define our notion of viscosity solution testing with $N-$dimensional functions as usual. 

	\begin{definition}\label{defsol.44}
		A bounded upper semicontinuous function $u \colon \R^N \to \R$  
		is a viscosity subsolution to the Dirichlet problem \eqref{convex-33} if
		$u \leq g$ in $\R^N\setminus\overline{\Omega}$ and if 
		for each $\delta > 0$ and $\phi \in C^2(\R^N)$ such that $x_0$ is a maximum 
		point of $u - \phi$ in $B_\delta(x_0)$, then
		\begin{equation*}
			\begin{array}{ll}
				\displaystyle
					E_\delta(u^g, \phi, x_0) \leq 0 & \quad \mbox{if} \ x_0 \in \Omega, \\[6pt]
				\displaystyle 
					\min \left\{ E_\delta(u^g, \phi, x_0), u(x_0) - g(x_0) \right\} \leq  0 
					& \quad \mbox{if} \ x_0 \in \partial \Omega.
			\end{array}
		\end{equation*}
		When $u$ is not upper semicontinuous we ask for the upper semicontinuous envelope of $u$ to be a subsolution.

		In an analogous way, we define viscosity supersolutions (reversing the inequalities) 
		and viscosity solutions (asking that $u$ is both a supersolution and a subsolution) to \eqref{convex-33}.
	\end{definition}

		In this work, we only consider bounded continuous exterior data, 
		but straightforward extensions to unbounded exterior data  are possible under certain growth 
		condition at infinity.  In fact, in the proof of the existence (we will use the Perron method) and
		as in 
		Theorems 1  and 2 of \cite{Barles-Imbert},  
		we can assume, for instance, 
		$|g(x)|\leq C(1+R(x))$ where $R(x)$ is a fixed positive function so that the operator is well defined
		(for example, a linear grow of $R$ at infinity suffices).

	\subsection{Attainability of the exterior datum.}
		{Now we follow ideas from \cite{BarChasImb}. 
		The main difference with respect to \cite{BarChasImb} is the 
		estimate after equation 
		\eqref{eq:attainablility5} below, 
		where the strict convexity of the domain plays a 
		crucial role, see  Example 5.1 below for non convex domain 
		where a loss of the boundary condition occurs.}
	
		We first prove that the exterior datum is attained in a classical continuous way.

	\begin{theorem}\label{condicion-borde.77} 
		Assume that  $f\in C(\overline{\Omega})$ and 
		$g\in C(\overline{\R^N\setminus\Omega})$ are bounded
		and that $\Omega$ is a {bounded} strictly convex $C^2-$domain. 
		Let $u,v\colon \mathbb{R}^N\to\mathbb{R}$ 			
		be  viscosity sub and supersolution of \eqref{convex-33}, 
		in the sense of Definition \ref{defsol.44}, respectively. Then,
		\begin{enumerate}[(i)]
			\item $u\le g$ on $\partial \Omega$; 
			
			\medskip
			
			\item $v\ge g$ on $\partial \Omega$.
		\end{enumerate}
	\end{theorem}

	\begin{proof}		
		We begin by proving {\it (i)}. Suppose by contradiction that 
		there is $x_0\in\partial\Omega$ such that 
		\[
			\mu \colon=u(x_0)-g(x_0)>0.
		\] 
		Hence, we have that $u^g(x_0)=u(x_0)$.
		Since $g$ is continuous, there is $R_0 > 0$ such that 
		\begin{equation}\label{eq:attainablility1}
			|g(x_0)-g(y)| \leq \frac{\mu}4
			\quad \forall y\in B(x_0,2R_0)\cap (\RR^N\setminus\Omega).
		\end{equation}
	 	We may with  no loss of
	 	generality assume that $R_0<\max\{\|x-y\|\colon x,y\in\overline{\Omega}\}.$
	 	
		We now introduce two auxiliary functions:
		\begin{itemize}
			\item $a\colon\RR^N\to \RR$, a smooth bounded function such that 
				$a(0)=0,$ $a(y)>0$ if $y\neq0,$ 
				\[
					\liminf\limits_{|y|\to\infty }a(y)>0
				\] 
				and $D^2 a$ is bounded;
			\item $b\colon\RR\to \RR$, a smooth bounded and increasing function 
				which is concave in $(0,+\infty),$ and \\
				such that $b(0)=0,$ $b(t)>-\frac{\mu}4$ in $\RR,$  
				$b^{\prime}(0)=k_1$ and $b^{\prime\prime}(0)=-k_2$ with $k_1,k_2>0.$ 
		\end{itemize} 
	
		Next, we use these two functions to define for any $\varepsilon>0$ the penalized test function
		\[
			\omega_\varepsilon (y)\coloneqq\frac{a(y-x_0)}{\varepsilon}+
			b\left(\frac{\dist(y)}{\varepsilon}\right)
		\]
		(recall that $\dist$ is a smooth extension of the signed distance to the boundary, $\partial \Omega$). 
		
		Thus 
		\[
			\Psi_\varepsilon(y)= u^g(y)-\omega_\varepsilon(y)
		\]
		is upper semicontinuous for any  $\varepsilon$ small. 
		Then, for any  $\varepsilon$ small, $\Psi_\varepsilon$ attains a global maximum at a point
		$x_\varepsilon.$
		Therefore, we have 
		\begin{equation}\label{eq:attainablility2}
			u^g(x_\varepsilon)-\omega_\varepsilon(x_\varepsilon)\geq 
			u^g(x_0)-\omega_\varepsilon(x_0)=u^g(x_0)
		\end{equation}
		and hence,
		\begin{equation}\label{eq:attainablility2.5}
	 		\frac{a(x_\varepsilon-x_0)}{\varepsilon}\leq 
	 			u^g(x_\varepsilon)-u(x_0)-b\left(\frac{\dist(x_\varepsilon)}{\varepsilon}\right)
	 		\leq
	 		u^g(x_\varepsilon)-u(x_0)+\|b\|_\infty.
		\end{equation}
		From here, we get that 
		\begin{equation}\label{eq:attainablility3}
			x_\varepsilon\to x_0 \text{ as } \varepsilon \to 0,
		\end{equation}
		In particular, $x_\varepsilon\in B(x_0,2R_0)$ for any $\varepsilon$ small enough. 
		Now, using again the properties of $a$ and $b$ we get
		\begin{equation}\label{eq:attainablility2.9}
			g(x_0)+\mu=u(x_0)\leq u^g(x_{\varepsilon})- 
			\omega(x_{\varepsilon})
			\leq u^g(x_{\varepsilon})+\frac{\mu}{4}.
		\end{equation}
		Therefore $x_\varepsilon\in\overline{\Omega}.$ 
		Then,  $\dist(x_\varepsilon)\ge0$ and
		\[
			0\le b\left(\frac{\dist(x_{\varepsilon} )}{\varepsilon}\right).
		\]

		On the other hand, since $a$ is non-negative, by \eqref{eq:attainablility2}, 
		we have
		\[
			0\le b\left(\frac{\dist(x_{\varepsilon} )}{\varepsilon}\right)
			\le u^g(x_{\varepsilon})-u(x_{0}).
		\]
		Hence, since $u^g$ is upper semicontinuous, $u^g(x_0)=u(x_0),$
		we obtain
		\begin{equation}\label{eq:attainablility3.5}
			b\left(\frac{\dist(x_{\varepsilon} )}{\varepsilon}\right)
			\to 0 \mbox{ as } \varepsilon\to 0.
		\end{equation}
		
		By \eqref{eq:attainablility2.5} and \eqref{eq:attainablility3.5},
		using the properties of $a$ and $b$, we have that
		\begin{equation}\label{eq:attainablility4}
			\frac{a\left(x-x_0\right)}{\varepsilon}
			\to 0,\quad
			\frac{\dist(x_{\varepsilon} )}{\varepsilon}
			\to 0 \quad \mbox{ and } \quad u^g(x_{\varepsilon})\to u(x_0)
		\end{equation}
		as $\varepsilon\to0,$ since $u^g$ is upper semicontinuous, $u^g(x_0)=u(x_0)$
		and $x_\varepsilon\in\overline{\Omega}$ for any $\varepsilon$ small enough.

		Since $u^g(x_\varepsilon)\to u(x_0)=g(x_0)+\nu$ and $g$ is continuous,
		if $x_\varepsilon\in\partial\Omega$ then $u(x_\varepsilon)>g(x_\varepsilon)$ 
		for any $\varepsilon$ small enough. Now, using that $u$ is a viscosity subsolution  
		of \eqref{convex-33} in $\Omega$ in the sense of Definition \ref{defsol.44}, we have that
		\begin{equation}\label{eq:attainablility5}
			E_\delta(u^g, \omega_\varepsilon, x_\varepsilon)\le0.
		\end{equation}

		\medskip
		
		\noindent{\it Case 1:} $x_\varepsilon\in\Omega.$ 
		
		Let $\bar x_\varepsilon\in\partial\Omega$ be such that 
		$\delta_\varepsilon\coloneqq\dist(x_{\varepsilon} )
		=\|x_\varepsilon-\bar x_\varepsilon\|$ and let
		$z_\varepsilon\coloneqq\tfrac{\bar x_\varepsilon-x_\varepsilon}{\|x_\varepsilon-\bar x_\varepsilon\|}.$ 
		Then, by \eqref{eq:attainablility5}, we get
		\[
			E_{z_{\varepsilon},\delta_\varepsilon}(u^g, \omega_\varepsilon, x_\varepsilon)\ge0
		\]
		and therefore
		\begin{equation}\label{eq:attainablility6}
			-\|f\|_\infty
			\le f(x_\varepsilon)\le I^1_{z_\varepsilon, \delta_\varepsilon}(\omega_\varepsilon, x_\varepsilon)+
			I^2_{z_\varepsilon, \delta_\varepsilon}(u^g,x_\varepsilon).
		\end{equation}
			
		Without loss of generality we suppose that $0<\delta_\varepsilon<\varepsilon<R_0$ 
		(see \eqref{eq:attainablility1}) due to  \eqref{eq:attainablility3}
		and   \eqref{eq:attainablility4}.  
		
		Observe that, in this case, there is $d_\varepsilon\ge\delta_\varepsilon$ such that
		\begin{align*}
			&x_\varepsilon +tz_\varepsilon\in\Omega \quad\forall t\in(-d_\varepsilon,\delta_\varepsilon),\\
			&x_\varepsilon +tz_\varepsilon\not\in\Omega \quad\forall 
			t\not\in(-d_\varepsilon,\delta_\varepsilon).
		\end{align*}
		Now, let us analyze the integrals that appear in $E_{z_{\varepsilon},\delta_\varepsilon}(u^g, \omega_\varepsilon, x_\varepsilon)$, 
		we have
		\[
			I^1_{z_\varepsilon, \delta_\varepsilon}(\omega_\varepsilon, x_\varepsilon)+
			I^2_{z_\varepsilon, \delta_\varepsilon}(u^g, x_\varepsilon)\le  
			I^1_{z_\varepsilon, \delta_\varepsilon}(\omega_\varepsilon, x_\varepsilon)+
			J^1_{z_\varepsilon, \varepsilon}(u^g, x_\varepsilon)
			+J^2_{z_\varepsilon, \varepsilon}(g, x_\varepsilon)
			+J^3_{z_\varepsilon, \varepsilon}(u, x_\varepsilon)
		\]
		where 
		\begin{align*}
			&J^1_{z_\varepsilon, \varepsilon}(u^g, x_\varepsilon)\coloneqq
			\int_{C_{\varepsilon}}\dfrac{u^g(x_\varepsilon+t z_\varepsilon)-
				u(x_\varepsilon)}{|t|^{1+2s}}\,dt, \quad \mbox{with } 
				C_{\varepsilon}\coloneqq(-\infty,-\varepsilon)\cup(\varepsilon,\infty),\\[6pt]
			&J^2_{z_\varepsilon, \varepsilon}(g, x_\varepsilon)\coloneqq
			\int_{\delta_\varepsilon}^{\varepsilon}
				\dfrac{g(x_\varepsilon+tz_\varepsilon)-u(x_\varepsilon)}{|t|^{1+2s}}\,dt, \\[6pt]
			&J^3_{z_\varepsilon, \delta}(u, x_\varepsilon)\coloneqq
			\int_{-\varepsilon}^{-\delta_\varepsilon}
				\dfrac{u(x_\varepsilon+t z_\varepsilon)-
				u(x_\varepsilon)}{|t|^{1+2s}}\,dt.
		\end{align*}
		
		Since $g$ is bounded and $u$ is upper semicontinuous in $\overline{\Omega}$, we have that there 
		is a positive constant $C$ independent of
		$\varepsilon$ such that
		\begin{equation}\label{eq:attainablility7}
			|J^1_{z_\varepsilon, \varepsilon}(u^g, x_\varepsilon)|\le C \varepsilon^{-2s}.
		\end{equation}
		
		On the other hand, by \eqref{eq:attainablility1} and 
		\eqref{eq:attainablility2.9}, there is a positive 
		constant $K$ independent of $\varepsilon$  such that
		\[
			J^2_{z_\varepsilon,\delta}(g, x_\varepsilon)\le
			\int_{\delta_\varepsilon}^\varepsilon
			\frac{g(x_0)+\tfrac{\mu}4-u(x_\varepsilon)}{|t|^{1+2s}}dt
			\le -\frac12\int_{\delta_\varepsilon}^\varepsilon
			\frac{dt}{|t|^{1+2s}}.
		\]
		Therefore there is a positive 
		constant $K$ independent of $\varepsilon$  such that
		\begin{equation}\label{eq:attainablility8}
			J^2_{z_\varepsilon,\delta}(g, x_\varepsilon)\le 
			-K\mu\left(\delta_\varepsilon^{-2s}-\varepsilon^{-2s}\right).
		\end{equation}
		
		By the properties of $a$ and $b,$ we have
		\begin{equation}\label{eq:propayb}
			\begin{aligned}
				D\omega_\varepsilon(x_\varepsilon)&=\dfrac{o(1)}{\varepsilon}+
					\dfrac{k_1+o(1)}{\varepsilon}	D\dist(x_\varepsilon );\\[6pt]
				D^2\omega_\varepsilon(x_\varepsilon)&=\dfrac{O(1)}{\varepsilon}+
					\dfrac{k_1+o(1)}{\varepsilon}	D^2\dist(x_\varepsilon )
					-\dfrac{k_2+o(1)}{\varepsilon^2}D\dist(x_\varepsilon )\otimes
					D\dist(x_\varepsilon );
			\end{aligned}
		\end{equation}
		and therefore there is a positive 
		constant $C$ independent of $\varepsilon$ and $\delta$ such that
		\begin{equation}\label{eq:attainablility9}
			J^3_{z_\varepsilon,\delta}(u, x_\varepsilon)\le
			J^3_{z_\varepsilon,\delta}(\omega_\varepsilon, x_\varepsilon)\le 
			C\kappa(\varepsilon,s)\quad\mbox{where }\kappa(\varepsilon,s)\coloneqq
			\begin{cases}
				\dfrac{|\delta_\varepsilon^{1-2s}
				-\varepsilon^{1-2s}|}{\varepsilon} &\mbox{if }s\neq\dfrac12,\\[12pt]
				-\dfrac{\ln\left(\delta_\varepsilon / \varepsilon \right)}{\varepsilon}
				&\mbox{if }{s=\dfrac12},\\[5pt]
			\end{cases} 
		\end{equation}
		and
		\begin{equation}\label{eq:attainablility10}
			|I^1_{z_\varepsilon, \delta_\varepsilon}(\omega_\varepsilon, x)|\le \dfrac{C}{\varepsilon^2}
			\delta_\varepsilon^{2-2s}.
		\end{equation}
		
		Then, by \eqref{eq:attainablility6}, \eqref{eq:attainablility7},
		\eqref{eq:attainablility8}, \eqref{eq:attainablility9}, and \eqref{eq:attainablility10}, we get
		\[
			-\|f\|_\infty
			\le \frac{C}{\delta_\varepsilon^{2s}}
			\left[
				\left(\dfrac{\delta_\varepsilon}{\varepsilon}\right)^{2}+
				(1+K\mu)\left(\dfrac{\delta_\varepsilon}{\varepsilon}\right)^{2s}
				+
				\kappa(\varepsilon,s)\delta_\varepsilon^{2s}-K\mu
			\right].
		\]
		Finally, from \eqref{eq:attainablility4}, taking $\varepsilon$ small enough we have
		\[
			-\|f\|_\infty
				\le \frac{C}{\delta_\varepsilon^{2s}}
				\left[
					\left(\dfrac{\delta_\varepsilon}{\varepsilon}\right)^{2}+
					(1+k\mu)\left(\dfrac{\delta_\varepsilon}{\varepsilon}\right)^{2s}+
					\kappa(\varepsilon,s)\delta_\varepsilon^{2s}-K\mu
				\right]
				<-\|f\|_\infty.
		\]
		and we have a contradiction.
		
		\medskip
		
		\noindent{\it Case 2:} $x_\varepsilon\in\partial\Omega.$ 
		
		In this case we take, $\delta_\varepsilon\coloneqq\|x_\varepsilon- x_0\|$ and let
		$z_\varepsilon\coloneqq\tfrac{x_0-x_\varepsilon}{\|x_0 - x_\varepsilon \|}.$ 
		Then by \eqref{eq:attainablility5}, we get
		\[
			E_{z_{\varepsilon},\delta_{\varepsilon}}(u^g, \omega_\varepsilon, x_\varepsilon)\ge0
		\]
		and therefore
		\begin{equation}\label{eq:attainablility66}
			-\|f\|_\infty
			\le f(x_\varepsilon)\le I^1_{z_\varepsilon, \delta_{\varepsilon}}
			(\omega_\varepsilon, x_\varepsilon)+
			I^2_{z_\varepsilon, \delta_{\varepsilon}}(u,x_\varepsilon).
		\end{equation}
		
		As in the previous case, without loss of generality we suppose that 
		$0<\delta_\varepsilon<\varepsilon<R_0.$
		
		Here, since $\Omega$ is strictly convex, we have that
		\begin{align*}
			&x_\varepsilon +tz_\varepsilon\in\Omega \quad\forall t\in(0,\delta_\varepsilon),\\[6pt]
			&x_\varepsilon +tz_\varepsilon\not\in\Omega \quad\forall t
			\not\in(0,\delta_\varepsilon).
		\end{align*} 
		Then, we compute
		\[
			I^1_{z_\varepsilon, \delta_\varepsilon}(\omega_\varepsilon, x)+
			I^2_{z_\varepsilon, \delta_\varepsilon}(u^g, x_\varepsilon)\le  
			I^1_{z_\varepsilon, \delta_\varepsilon}(\omega_\varepsilon, x_\varepsilon)+
			J^1_{z_\varepsilon,\varepsilon}(u^g, x_\varepsilon)
			+J^2_{z_\varepsilon, \delta_\varepsilon}(g, x_\varepsilon)
		\]
		where
		\begin{align*}
			&J^1_{z_\varepsilon, \varepsilon}(u^g, x_\varepsilon)\coloneqq
			\int_{C_{\varepsilon}}\dfrac{u^g(x_\varepsilon+t_\varepsilon z_\varepsilon)-
				u(x_\varepsilon)}{|t|^{1+2s}}\,dt,
				 \quad { \mbox{where } 
				C_{\varepsilon}\coloneqq(-\infty,-\varepsilon)\cup(\varepsilon,\infty)},\\[6pt]
			&J^2_{z_\varepsilon, \delta_\varepsilon}(g, x_\varepsilon)\coloneqq
			\int_{A_\varepsilon}
				\dfrac{g(x_\varepsilon+t_\varepsilon z_\varepsilon)-u(x_\varepsilon)}{|t|^{1+2s}}\,dt, 
				\quad \mbox{where } A_{\varepsilon}
				=(-\varepsilon,-\delta_{\varepsilon})\cup(\delta_\varepsilon,\varepsilon).
		\end{align*}
		
		As in the above case,  there are two positive constants $C$ and $K$ independent of
		$\varepsilon$  such that
		\begin{align*}
			|J^1_{z_\varepsilon, \delta_\varepsilon}(u^g, x_\varepsilon)|&\le C\varepsilon^{-2s},\\[6pt]
			J^2_{z_\varepsilon,\delta_\varepsilon}(g, x_\varepsilon)&\le 
			-K\mu\left(\delta_\varepsilon^{-2s}-\varepsilon^{-2s}\right),\\[6pt]
			|I^1_{z_\varepsilon, \delta_\varepsilon}(\omega_\varepsilon, x)|&\le \dfrac{C}{\varepsilon^2}
			\delta_{\varepsilon}^{2-2s}.
		\end{align*}

		Then, we get
		\[
			-\|f\|_\infty
			\le \frac{1}{\delta_\varepsilon^{2s}}
			\left\{C\left[
				\left(\dfrac{\delta_{\varepsilon}}{\varepsilon}\right)^2+
				(1+\mu)\left(\dfrac{\delta}{\varepsilon}\right)^{2s}
			\right]-K\mu\right\}.
		\]
		Finally, from \eqref{eq:attainablility4}, taking $\varepsilon$ small enough we have
		\[
		-\|f\|_\infty
			\le \frac{1}{\delta_\varepsilon^{2s}}
			\left\{C\left[
				\left(\dfrac{\delta_{\varepsilon}}{\varepsilon}\right)^2+
				(1+\mu)\left(\dfrac{\delta}{R_0}\right)^{2s}
			\right]-K\mu\right\} <-\|f\|_\infty.
		\]
		and we have again a contradiction.

		\medskip
		
		Arguing similarly, using that $\Omega$ is strictly convex and 
		doing some simple changes in $\omega_\varepsilon$, we obtain (ii).
		\end{proof}
		

	\subsection{Comparison principle.}
	Now, we prove a comparison principle for \eqref{convex-33}.

	\begin{theorem}\label{teo:cp}
		Assume that $f\in C(\overline{\Omega})$ and 
		$g\in C(\overline{\R^N\setminus\Omega})$ are 
		bounded and that $\Omega$ is a {bounded} strictly convex $C^2-$domain. 
		Let $u,v\colon \mathbb{R}^N\to\mathbb{R}$ 			
		be a viscosity sub and supersolution of \eqref{convex-33}, 
		in the sense of Definition \ref{defsol.44}, then $$u\le v$$ in $\mathbb{R}^N$.
	\end{theorem}
	\begin{proof}
		Define
		\[
			M\coloneqq\sup \Big\{ u(x)-v(x)\colon x\in \overline{\Omega} \Big\}.
		\]
		As usual, we argue by contradiction, that is, we assume that $M>0.$
		Since $u$ and $v$ are upper and lower semicontinuous functions, 
		\[
			S\coloneqq \sup\Big\{u(x)\colon x\in\overline{\Omega} \Big\}-\inf 
			\Big\{v(x)\colon x\in\overline{\Omega} \Big\}<\infty.
		\]
		
		For any $\varepsilon>0,$ we define
		\[
			\Psi_\varepsilon(x,y)\coloneqq u(x)-v(y)-\dfrac{\|x-y\|^2}\varepsilon.
		\] 
		Observe that
		\begin{equation}\label{eq:cp1}
			M\le M_\varepsilon\coloneqq\sup \Big\{\Psi_\varepsilon(x,y)\colon (x,y)\in\Omega\times\Omega\Big\}
			\le S.
		\end{equation}
		Moreover, $M_{\varepsilon_1}\le M_{\varepsilon_2}$ for all $\varepsilon_1\le\varepsilon_2.$ Then, there
		exists the limit
		\begin{equation}\label{eq:cp2}
			\lim_{\varepsilon \to 0^+} M_\varepsilon=\overline{M}.
		\end{equation}
		
		On the 	other hand, since $u$ and $-v$ are upper semicontinuous functions, for any $\varepsilon$,
		$\Psi_\varepsilon$ is an upper semicontinuous function. Thus, there is 
		$(x_\varepsilon,y_\varepsilon)\in\overline{\Omega}\times\overline{\Omega}$ such that
		\begin{equation}\label{eq:xeye}
			M_\varepsilon=\Psi_\varepsilon(x_\varepsilon,y_\varepsilon).
		\end{equation}
		Observe that
		\[
			M_{2\varepsilon}\ge \Psi_{2\varepsilon}(x_\varepsilon,y_\varepsilon)=
			\Psi_{\varepsilon}(x_\varepsilon,y_\varepsilon)+
			\dfrac{\|x_\varepsilon-y_{\varepsilon}\|^2}{2\varepsilon}=M_\varepsilon+
			\dfrac{\|x_\varepsilon-y_{\varepsilon}\|^2}{2\varepsilon}
		\]
		implies
		\[
			\dfrac{\|x_\varepsilon-y_{\varepsilon}\|^2}{\varepsilon}\le 
			2(M_{2\varepsilon}-M_{\varepsilon})\to 0
		\]
		as $\varepsilon\to 0^+.$ Therefore, we have
		\begin{equation}\label{eq:cp3}
			\lim_{\varepsilon \to 0^+} \dfrac{\|x_\varepsilon-y_{\varepsilon}\|^2}{\varepsilon}=0.
		\end{equation}
		
		Since $\overline{\Omega}$ is compact, extracting a subsequence if necessary, we can assume that
		\begin{equation}\label{eq:cp4}
			\lim_{\varepsilon \to 0^+} (x_\varepsilon,y_\varepsilon)\to ({\bar x},{\bar y})\in
			\overline{\Omega}\times\overline{\Omega}.
		\end{equation}
		Moreover, by \eqref{eq:cp3}, ${\bar x}={\bar y}$ and
		\[
			M\le \overline{M}=\lim_{\varepsilon\to0^+}\Psi_\varepsilon(x_\varepsilon,y_\varepsilon)
			\le\limsup_{\varepsilon\to0^+} u(x_\varepsilon)-v(x_\varepsilon)\le
			u({\bar x})-v({\bar x})\le M.
		\]
		Thus $M=u({\bar x})-v({\bar x}),$  and by Theorem \ref{condicion-borde.77}, 
		${\bar x}\in\Omega.$ Consequently, we may assume (without loss of generality) that
		\[
			d_\varepsilon=\min \Big\{\dist(x_\varepsilon), 
			\dist(y_\varepsilon )\Big\}>\frac{\dist({\bar x})}2>0
		\]
		provided $\varepsilon$ is small enough.
		
		On the other hand, by \eqref{eq:xeye}, for any $w\in \RR^N$ such that 
		 we have $(x_\varepsilon+w,y_\varepsilon+w)\in\overline{\Omega}\times\overline{\Omega}$ we have that
		\begin{equation}\label{eq:uvpositive}
			0\le u(x_\varepsilon)-u(x_\varepsilon+w)
			-v(y_\varepsilon)+v(y_\varepsilon+w).
		\end{equation}
		So
		\begin{align*}
			\phi_\varepsilon(x)&\coloneqq 
			v(y_\varepsilon)
			+\dfrac{\|x-y_\varepsilon\|^2}{\varepsilon},\\[6pt]
			\varphi_\varepsilon(y)&\coloneqq
			u(x_\varepsilon)
			-\dfrac{\|x_\varepsilon-y\|^2}{\varepsilon} 
		\end{align*}
		are test functions for $u$ and $v$ at $x_\varepsilon$ and $y_\varepsilon,$ respectively. Then
		\[
			E_\delta (u^g, \phi_\varepsilon, x_\varepsilon)\leq 0,\quad 
			\mbox{and}\quad E_\delta (v_g, \varphi_{\varepsilon}, y_\varepsilon)\geq 0 
		\]
		for all $\delta\in(0,d_\varepsilon).$
		Therefore, by the definition of $E_\delta,$ 
		for each $h>0$ there exists $z_{\varepsilon,h}\in \mathbb{S}^{N-1} $ such that
		\begin{equation}\label{EE}
			E_{(z_{\varepsilon,h}),\delta} (u^g, \phi_\varepsilon, x_\varepsilon)\geq 0 ,\quad 
			\mbox{and}
			\quad  E_{(z_{\varepsilon,h}),\delta}(v_g, \varphi_{\varepsilon}, y_\varepsilon)\leq h
		\end{equation}
		for any $\delta\in(0,d_\varepsilon).$

		Now, our goal is to obtain upper estimates for each term in the difference
		\[
			E_{(z_{\varepsilon,h}),\delta}(u^g, \phi_\varepsilon, x_\varepsilon)- 
			E_{(z_{\varepsilon,h}),\delta} (v_g, \varphi_{\varepsilon}, y_\varepsilon)
		\]
		for $0<\delta<\tfrac{\dist(\bar x )}2.$ We can 
		assume that $$z_{\varepsilon,h} \to z_0,$$
		taking a subsequence 
		if necessary. 		
		
		Let us write
		\begin{equation}\label{EE.89}
			\begin{array}{l}
				\displaystyle 
					E_{(z_{\varepsilon,h}),\delta} (u^g, \phi_\varepsilon, x_\varepsilon)
						=I_1(\delta,\varepsilon,h)+I_2(\delta,\varepsilon,h)+
						I_3(\varepsilon,h)-f(x_\varepsilon) \\[15pt]
				\displaystyle 
					 E_{(z_{\varepsilon,h}),\delta} (v_g, \varphi_{\varepsilon}, y_\varepsilon)
					=J_1(\delta,\varepsilon,h)+J_2(\delta,\varepsilon,h)+
					J_3(\varepsilon,h)-f(y_\varepsilon),
			\end{array}
		\end{equation}
		where
		\[
			\begin{array}{ll}
				\displaystyle 
				I_1(\delta,\varepsilon,h)\coloneqq \int_{-\delta}^\delta 
					\frac{
						\phi_{\varepsilon}(x_{\varepsilon}+t z_{\varepsilon,h})-
							\phi_{\varepsilon}(x_{\varepsilon})}{|t|^{1+2s}}dt,
					& 
					\displaystyle
					J_1(\delta,\varepsilon,h)\coloneqq \int_{-\delta}^\delta \frac{\varphi_{\varepsilon}
					(y_{\varepsilon}+t z_{\varepsilon,h})-
					\varphi_{\varepsilon, \mu}(y_\varepsilon)}{|t|^{1+2s}}dt,
					\\[15pt]
				\displaystyle 
					I_2(\delta,\varepsilon,h)
					\coloneqq \int_{A_\delta^{z_{\varepsilon,h}}(x_\varepsilon)} 
					\frac{u(x_{\varepsilon}+t z_{\varepsilon,h})-u(x_\varepsilon)}{|t|^{1+2s}}dt,
				& \displaystyle
					J_2(\delta,\varepsilon,h)
					\coloneqq\int_{A_\delta^{z_{\varepsilon,h}}(y_\varepsilon)} 
				\frac{v(y_\varepsilon+t z_{\varepsilon,h})-v(y_\varepsilon)}{|t|^{1+2s}}dt,
				\\[15pt]
				\displaystyle
				I_3(\varepsilon,h)
				\coloneqq \int_{\R\setminus L_{z_{\varepsilon,h}}(\bar x)} 
				\frac{g(x_{\varepsilon}+t z_{\varepsilon,h})-u(x_\varepsilon)}{|t|^{1+2s}}dt,
				&
				\displaystyle
				J_3(\varepsilon,h)
				\coloneqq
				\int_{\R\setminus
				L_{z_{\varepsilon,h}}(y_\varepsilon)} \frac{g(y_\varepsilon 
				+t z_{\varepsilon,h})-v(y_\varepsilon)}{|t|^{1+2s}}dt,
			\end{array}
		\]
		where $A_\delta^{z}(x)\coloneqq \{t\in \mathbb{R}\colon x+tz\in\Omega\} 
		\setminus (-\delta,\delta).$
		
		We first observe that there is a positive constant $C$ independent of $\delta,$ 
		$\varepsilon$ and $h$ such that
		\[
			\max \Big\{|I_1(\delta,\varepsilon,h)|, |J_1(\delta,\varepsilon,h)|\Big\}
			\le C\dfrac{\delta^{2-2s}}{\varepsilon}.
		\]

		Let $L_z(x)$ denote the line that passes trough $x$ with direction $z$, that is,
		$L_z (x) := x+tz$, $t \in \mathbb{R}$. 
		For the estimate of $I_2(\delta,\varepsilon,h)-J_2(\delta,\varepsilon,h),$ 
		we use that $z_{\varepsilon,h} \to z_0$ and that $\Omega$ is strictly convex to get,
		\[
			\mathbf{1}_{L_{z_{\varepsilon,h}}(y_\varepsilon)\cap \Omega}, \mathbf{1}_{L_{z_{\varepsilon,h}}
			(x_\varepsilon)\cap \Omega} \to 
			\mathbf{1}_{L_{z_0}(\bar x) \cap \Omega}\quad \mbox{ a.e. as }\quad\varepsilon,h \to 0.
		\]
		Then, since $u$ and $v$ are bounded, 
		by \eqref{eq:uvpositive} and \eqref{eq:cp4}, 
		and the dominated convergence theorem, we have
		\[
			\limsup_{\delta,\varepsilon,h\to0}
			I_2(\delta,\varepsilon,h)-J_2(\delta,\varepsilon,h)\le 0.
		\]

		Finally, we observe that, using again that  $z_{\varepsilon,h} \to z_0$, the strictly convexity of $\Omega$, 
		and the dominated convergence theorem, we have
		\[
			\mathbf{1}_{L_{z_{\varepsilon,h}}(x_\varepsilon) \cap (\mathbb{R}^N \setminus \Omega)}, 
			\mathbf{1}_{L_{z_{\varepsilon,h}}(y_\varepsilon)\cap (\mathbb{R}^N \setminus \Omega)} 
			\to \mathbf{1}_{L_{z_0}(\bar x)\cap (\mathbb{R}^N \setminus \Omega)} \quad\mbox{ a.e. as }\quad\varepsilon,h \to 0,
		\]
		and hence, using that $g$ is a bounded continuous function, we obtain
		\[
			\lim_{\varepsilon,h\to0}J_3(\varepsilon,h)-I_3(\varepsilon,h)
			\to -M \int_{L_{z_0}(\bar x)\cap (\mathbb{R}^N \setminus \Omega)} \frac{dt}{|t|^{1+2s}}, \quad \mbox{ a.e.}
		\]

		Therefore, 
		from our previous estimates we obtain, letting first $\delta \to 0$, then $\varepsilon \to 0 $, 
		and $h \to 0,$ we get
		\begin{equation}\label{eq:final}
		\begin{array}{l}
		\displaystyle
			I_1(\delta,\varepsilon,h)-J_1(\delta,\varepsilon,h)+I_2(\delta,\varepsilon,h)-
			J_2(\delta,\varepsilon,h)+
						I_3(\varepsilon,h)-J_3(\varepsilon,h)+f(y_\varepsilon)-f(x_\varepsilon) \\[6pt]
						\displaystyle \qquad 
						\to -M \int_{L_{z_0}(\bar x)\cap (\mathbb{R}^N \setminus \Omega)}
						\frac{dt}{|t|^{1+2s}}.
						\end{array}
		\end{equation}

		On the other hand, from \eqref{EE}, we get
		{\begin{align*}
		- h &\leq E_{(z_{\varepsilon,h}),\delta}(u^g, \phi_\varepsilon, x_\varepsilon)- 
			E_{(z_{\varepsilon,h}),\delta} (v_g, \varphi_{\varepsilon}, y_\varepsilon)
			\\[6pt]
			&=I_1(\delta,\varepsilon,h)-J_1(\delta,\varepsilon,h)+I_2(\delta,\varepsilon,h)-
			J_2(\delta,\varepsilon,h)+
						I_3(\varepsilon,h)-J_3(\varepsilon,h)+f(y_\varepsilon)-f(x_\varepsilon)
		\end{align*}}
		and therefore, letting first $\delta \to 0$, then $\varepsilon \to 0 $, and  
		$h \to 0$ we conclude that
		\[
			0\le -M \int_{L_{z_0}(\bar x)\cap (\mathbb{R}^N \setminus \Omega)} \frac{dt}{|t|^{1+2s}}<0,
		\]
		a contradiction.
\end{proof}

\subsection{Existence and uniqueness of a solution}

	Now our goal is to show existence and uniqueness of a solution to
	\begin{equation} \label{convex-330} 
		\begin{cases}
			\Lambda_1^s u (x) = 0 & x \in \Omega, \\[6pt]
					u (x) = g (x) & x \in \mathbb{R}^N\setminus \Omega,
		\end{cases}
	\end{equation}
	using Perron's method.

	\begin{theorem}\label{EyU.77.99} 
		Assume that $g\in C(\R^N\setminus\overline{\Omega})$ is bounded 
		and $\Omega$ is a {bounded} strictly convex $C^2-$domain.  
		Then, there is a unique viscosity solution $u$ to 
		\eqref{convex-330}, in the sense of Definition \ref{defsol.44}.
		This solution is continuous in $\overline{\Omega}$ and the datum $g$ is taken with continuity, 
		that is, $u|_{\partial \Omega} = g|_{\partial \Omega}$.
	\end{theorem}

	\begin{proof} To obtain the existence of 
		a viscosity solution to our problem \eqref{convex-330}, 
		we again use ideas from \cite{BarChasImb}

		The existence of viscosity subsolution and supersolution 
		of \eqref{convex-330} in $\Omega$, in the sense of Definition \ref{defsol.44}, 
		follows easily taking $\pm  \|g\|_\infty $ 
		(here we are using that $g$ is bounded). 
		 
		We now consider a one-parameter family of continuous functions 
		$\psi_{\pm}^k\colon \R^N \to \R$ such that 
		\[ 
			\psi_+^k \geq \psi_-^k \text{ in } \R^N \text{ and }
			 \psi^k_{\pm} = g \text{ in } \R^N\setminus\Omega.
		\]
		Then for all $k\in\N$ we consider the obstacle problem 
		\begin{equation}\label{obstacle}
			H_k(x, u)\coloneqq \min \left\{ u(x) - \psi_-^k(x), 
			\max \{ u(x)-\psi_+^k(x), -\Lambda_1^s u (x)  \} \right\} =0, 
			\quad x \in \R^N,
		\end{equation}
		which is degenerate elliptic (that is, it satisfies the general assumption 
		$(E)$ of \cite{Barles-Imbert}). That also has  $\pm  \|g\|_\infty $ as viscosity 
		supersolution and  subsolution ( independent on the $L^\infty$ 
		bounds of $\psi^k_{\pm}$). 
		Then, in view of the general Perron's method given in \cite{Barles-Imbert} 
		for problems in $\R^N$, since condition $(E)$  and comparison holds, 
		we conclude the existence of a continuous bounded
		viscosity solution $u_k$ to~\eqref{obstacle} for each $k.$ 
		In addition, 
		\begin{equation}\label{obstacle_aux}
			\|u_k\|_\infty\le \|g\|_\infty \text{ and } u_k=g \text{ in }
			\R^N\setminus\Omega
		\end{equation}
		for all $k$. 

		Then, we consider $\psi_{\pm}^k$ in such a way $\psi_{\pm}^k(x) \to \pm \infty$ as $k \to \infty$ 
		for all $x \in \Omega$ and denoting
		\begin{equation*}
			\bar u(x) = \limsup \limits_{k \to \infty, y \to x} u_k(y); \quad 
			\underline u(x) = \liminf \limits_{k \to \infty, y \to x} u_k(y),
		\end{equation*}
		which are well defined for all $x \in \R^N.$
		Thus, we clearly have that 
		\[ 
			\bar u \geq \underline u \text{ in }\R^N.
		\]
		Thus, since $\psi^k_{\pm} = g$ in $\R^N\setminus\Omega$
		for any $k\in \mathbb{N}$, we have 
		\[ 
			\underline u = \bar u = g \text{ in } \R^N\setminus\Omega. 
		\]
		Moreover,  by the half-relax limite properties, see \cite{Barles-Imbert},
		$\bar u$ and $\underline u$ respectively viscosity sub and supersolution 
		to our problem \eqref{convex-330}. Thus, by comparison we get
		$$\bar u \leq \underline u$$ in $\R^N$, and therefore we conclude that $\bar u$ and $\underline u$
		coincide and that
		$$u:=\bar u = \underline u$$
		is a continuous viscosity solution that satisfies the boundary condition 
		in the classical sense.

		Uniqueness of solutions follows from the comparison principle.
\end{proof}

\section{$s-$convex functions}\label{sect.3}

	Our next aim is to show that $u$ is $s-$convex if only if 
	\[
		\Lambda^s_1 u(x)=\inf 
				\left\{ \int_{\mathbb{R}} 
					\frac{u(x+tz)-u(x)}{|t|^{1+2s}} 
					\, dt\colon z\in \mathbb{S}^{N-1}
				\right\}\ge0 \text{ in } \Omega
	\] 
	in the viscosity sense testing $u$ with $1-$dimensional functions in every segment inside $\Omega$.

	For this reason, we need to introduce a different definition of 
	viscosity subsolution of $\Lambda^s_1 u =0.$

	Let us start with the definition of viscosity solution of the fractional Dirichlet problem
	in one dimension,
	\begin{equation}
		\label{eq:sconvex.a}
		\left\{
			\begin{array}{ll}
				\Delta^s_1 w(t) = 0 &\text{in } (0,1),\\[6pt]
				 w(t)=g(t) &\text{in } \mathbb{R}\setminus(0,1),
			\end{array}
			\right.
	\end{equation}
	where $0<s<1,$ $g\in L_{loc}^\infty(\mathbb{R})\cap L_s(\mathbb{R})$ and 
	$\Delta^s_1$ denotes the fractional laplacian operators,
	\[
		\Delta^s_1 w(t)
		\coloneqq 
		\int_{\mathbb{R}} \frac{w(r)-w(t)}{|r-t|^{1+2s}} \, dr. 
	\]
	
	{A function $w\colon \mathbb{R}\to \mathbb{R}$
		\textit{viscosity subsolution (supersolution)} of \eqref{eq:sconvex.a} if its
		upper (lower) semicontinuous envelope $\tilde{w}$ ($\undertilde{w}$) 
		satisfies $\tilde{w}(t)\le  g(t)$ ($\undertilde{w}(t)\le  g(t)$) 
		for any $t\in\mathbb{R}\setminus(0,1);$ 
		and for any open interval $I\subset (0,1),$ any $t_0\in I$ and any test function 
		$\phi\in C^{2}(\mathbb{R})$ such that $\tilde{w}-\phi$
		($\undertilde{w}(t)\le  g(t)$)  attains a maximum (minimum) 
		at $t_0$ in $I,$ if we let
				\[
					\hat{\phi}(t)\coloneqq
					\begin{cases}
							\phi(t) &\text{if }t\in I,\\
							\tilde{w}(t) (\undertilde{w}(t)) 
							&\text{if }t\in \R\setminus I,
					\end{cases}
				\]
				 we have
				\[
					-\Delta^s_1 \hat{\phi}(t_0)\le (\ge) 0.
				\]
			
		A  function $w\colon \mathbb{R}\to \mathbb{R}$ is a {\it viscosity solution}
		of \eqref{eq:sconvex.a} if it is both a supersolution and a subsolution.
		}

	Given a function 
	\[
		u\in \mathcal{L}_s(\mathbb{R}^N)\coloneqq
			\{v\in L^\infty_{loc}(\mathbb{R}^N )\colon t\to v(x+tz)\in L_s(\mathbb{R})\,
			\forall x\in\Omega, \forall z\in\mathbb{S}^{N-1}\}	
	\]
		
		and two points $x,y\in \Omega$ with $x\neq y$, we 
		now introduce the definition of
		viscosity solution to
	\begin{equation}\label{eq:sconvex}
	\left\{
			\begin{array}{ll}
				\Delta^s_1 v(tx+(1-t)y)=0 & t\in (0,1);\\[6pt]
				 v(tx+(1-t)y)=u(tx+(1-t)y) & t\in\mathbb{R}\setminus (0,1).
			\end{array}
			\right.
	\end{equation}
	
	\begin{definition}
		A function $v$ is a \textit{viscosity subsolution} (\textit{supersolution}) of \eqref{eq:sconvex} if
		$$w(t)=v(tx+(1-t)y)$$ is a viscosity subsolution (supersolution) of 
		\[
		\left\{
			\begin{array}{ll}
				\Delta^s_1 w(t) = 0 & t \in (0,1),\\[6pt]
				 w(t)=v(tx+(1-t)y) & t\in \mathbb{R}\setminus (0,1).
			\end{array}
			\right.
		\]
		{ A function $v\colon \mathbb{R}\to \mathbb{R}$ is a {\it viscosity solution}
			of \eqref{eq:sconvex}  if it is both a supersolution and a subsolution.}
	\end{definition}
	
	Now we are in a position to rigorously state the definition of being $s-$convex.
	\begin{definition}
		Let $s\in(0,1),$
		a function $u\in \mathcal{L}_s(\mathbb{R}^N)$
		is said to be {\it $s-$convex} in $\Omega$ if
		for any two points $x,y \in \Omega$ such that the segment $[x,y]$ is contained in 
		$\Omega$ it holds that
		\begin{equation} \label{convexo-s.88}
			u(t x+(1-t)y) \leq v(t x+(1-t)y), \qquad \forall t \in (0,1)
		\end{equation}
		where $v$ is just the viscosity solution 
		of \eqref{eq:sconvex}. 
	\end{definition}
	
	Notice that in the two previous definitions we used $1-$dimensional test functions to test the solution $v$.

	Now, we are ready to state our second definition of being a solution to 
	$$\Lambda^s_1 u=\inf 
				\left\{ \int_{\mathbb{R}} 
					\frac{u(x+tz)-u(x)}{|t|^{1+2s}} 
					\, dt\colon z\in \mathbb{S}^{N-1}
				\right\} = 0.$$

	\begin{definition} \label{defi.2.5}
		A function $\mathfrak{u}\colon\mathbb{R}^N\to\mathbb{R}$  is a 
		\textit{viscosity subsolution} of 
		\begin{equation}\label{eq:Lu}
			\Lambda^s_1 \mathfrak{u}(x)=0 \text{ in } \Omega
		\end{equation} 
		if  for any $x\in \Omega,$ any  
		$z\in\mathbb{S}^{N-1},$ any open interval $I\ni0$ such that $x+tz\in\Omega$ for all $t\in I,$ 
		and any test function $\phi\in C^{2}(\mathbb{R})$ 
		such that $\phi(0)=\tilde{w}(0)$ and $\phi(t)\ge \tilde{w}(t)$ 
		($\phi(t)\le \undertilde{w}(t)$) in $I$ we have
		\[
			-\Delta^s_1 \hat{\phi}(0)\le 0 
		\]
		where 
		\[
			\hat{\phi}(t)=
					\begin{cases}
							\phi(t) &\text{if }t\in I,\\
							{ \tilde{w}(t) } &\text{if }t\in\R\setminus I,
					\end{cases}
		\]
		and $\tilde{w}(t)$ denotes the upper semicontinuous envelope of $w(t)=\mathfrak{u} (x+tz).$
		
		A function $\mathfrak{u}\colon\mathbb{R}^N\to\mathbb{R}$  is a 
		\textit{viscosity supersolution} of 
		\eqref{eq:Lu}
		if  for any $x\in \Omega,$ any $\phi \in C^{2}(\mathbb{R^N})$ 
		such that $\phi(x)= \undertilde{w}(x)$ and $\phi(y)\le \undertilde{w}(y)$ in $\mathbb{R}^N$ we have
$$
-\Delta^s_1 \hat{\phi}(x)\ge 0,
$$
and $\undertilde{w}(y)$ denotes the lower semicontinuous envelope of $\mathfrak{u}$ in $\mathbb{R}^N$.
		
		Finally, we say that $\mathfrak{u} $ a viscosity solution of \eqref{eq:Lu} when it is both a 
		viscosity subsolution and a viscosity supersolution of \eqref{eq:Lu}.
	\end{definition}
	
	Notice that in this definition we used $1-$dimensional test functions to test from above, while from below we test
	with $N-$dimensional functions. 
	Recall that in the introduction we discussed the differences between the two 
	definitions of viscosity solutions of \eqref{eq:Lu}. 
	
	From the previous definition, Definition \ref{defi.2.5}, of viscosity subsolution of \eqref{eq:Lu}, we deduce the next lemma that
	we state for future reference.
	
	\begin{lemma}\label{lemma:auxsconv1}
		A function $\mathfrak{u} \colon\mathbb{R}^N \mapsto \mathbb{R}$ is a viscosity subsolution 
		of \eqref{eq:Lu} in the sense of Definition \ref{defi.2.5} 
		if only if for any $x\in\Omega$ and any $z\in\mathbb{S}^{N-1}$ 
		we have that the function $w(t)= \mathfrak{u}(x+tz)$ is a viscosity subsolution of 
		\[
			\Delta^s_1 w(t)=0 \quad \text{ in } 
			I_z(x)\coloneqq\{t\in \mathbb{R}\colon x+tz\in\Omega\}.
		\]
		
	\end{lemma}
	
	Next, recalling our definition of being $s-$convex let us look at a simple example.
	
	\begin{example} Let us present a simple explicit example of a function that is $s-$convex in a $1-$dimensional
	interval.
		Let $u\colon \mathbb{R}\to \mathbb{R}$
		
		\begin{minipage}{0.4\textwidth}
				    \[
					     u(t)\coloneqq
					            \begin{cases}
					                -(1-t^2)^s & \text{if } t\in[-1,1],\\[6pt]
					                0 & \text{if } |t|>1.
					            \end{cases}
					\]
			\end{minipage}
			\begin{minipage}{0.4\textwidth}
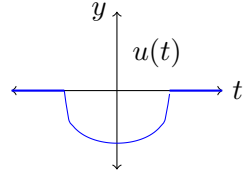
				\begin{center}
					    \begin{tikzpicture}
				            \draw[<->,scale=0.7] (-2,0) -- (2,0) node[right] {$t$};
				            \draw[<->,scale=0.7] (0,-1.5) -- (0,1.5) node[left] {$y$};
						    \draw[thick, blue,rounded corners=1mm,scale=0.7] (-2,0)--(-1,0);
					        \draw[thick, blue,rounded corners=1mm,scale=0.7] (1,0)--(2,0);
					        \draw (1,.5) node[left] {$u(t)$};
					        \draw[domain=-1:1,smooth,variable=\x,blue,scale=0.7] 
					        plot ({\x},{-1* (1-\x*\x)^(1/3)});
				        \end{tikzpicture}
			  	\end{center}
			\end{minipage}
			
		By \cite{Dyda},  we have that
				\[
					\Delta_1^s u(t)=\Gamma(2s+1) \quad \text{in }(0,1). 
				\]
		Then, given two points in $x,y\in (-1,1),$
		\[
			\Delta_1^s u(tx+(1-t)y)=|x-y|^{2s}\Gamma(2s+1)\quad \text{ for } t \in (0,1). 
		\]
		Thus, by the maximum principle, 
			if $v$ the viscosity solution of \eqref{eq:sconvex}
		then 
		\[ 
			u(tx+(1-t)y)\le v(tx+(1-t)y)
		\]
		for any $t\in(0,1).$ We conclude that $u$ is $s-$convex in $(-1,1)$.
		
		Notice that the same arguments show that every viscosity solution to 
		\[
			\Delta_1^s u(t)\geq 0 \quad \text{in } (-1,1)
		\]
		is $s-$convex in $(-1,1)$.
	\end{example}

	Now, we are ready to show one of the main results of this section.

	\begin{theorem}\label{theorem:Luconvex1}
		Let $u\in\mathcal{L}_s(\mathbb{R}^n)$ be a viscosity subsolution to
		\begin{equation}\label{eq:Lu.44}
			\Lambda^s_1 u(x)= 0 \quad \text{ in } \Omega,
		\end{equation}  
		in the sense of Definition \ref{defi.2.5} then $u$ is $s-$convex in $\Omega.$
	\end{theorem}
	
	\begin{proof}
		Let $x,y\in\Omega$ and $v$ be the viscosity solution of \eqref{eq:sconvex} 
		(that is well defined due to the fact that $u\in\mathcal{L}_s(\mathbb{R}^N)$). Notice that
		\[
			tx+(1-t)y=t |x-y|z+y \quad\text{with } z=\dfrac{x-y}{|x-y|}\in\mathbb{S}^{N-1}.	
		\]
		Then, by Lemma \ref{lemma:auxsconv1}, we have that $w(t)=u(y+t |x-y| z)$  is a viscosity 
		subsolution of  
		\[
			\Delta^s_1 w(t)=0 \quad \text{ in } (0,1).
		\] 
		Thus, by the maximum principle, we get
		\[
			u(tx+(1-t)y)= w(t)\le v(tx+(1-t)y)\quad \forall t\in(0,1).
		\]
		Therefore, we conclude that $u$ is $s-$convex.
	\end{proof}	
	
	Now, we prove the reciprocal result. 

	\begin{theorem}\label{theorem:Luconvex2}
		Let $u\colon\mathbb{R}^N\to\mathbb{R}$ 
		be $s-$convex in $\Omega.$  
		Then, $u$ is a viscosity subsolution to 
		\begin{equation}\label{eq:Lu.55}
			\Lambda^s_1 u(x)= 0 \quad \text{ in } \Omega
		\end{equation} 
		in the sense of Definition \ref{defi.2.5}.
	\end{theorem}
		
	\begin{proof}
		We argue by contradiction. Assume that there are $x\in \Omega,$ 
		$z\in\mathbb{S}^{N-1},$
		an open interval $I\ni0$ such that 
		$x+tz\in\Omega$ for all $t\in I,$ let $\tilde{w}(t) = u(x + tz)$,
		and let $\hat \phi\in C^{2}(\mathbb{R})$ be 
		a test function such that $\hat\phi(0)=\tilde{w}(0)$ and 
		$\phi(t)\ge \tilde{w}(t)$ 
		in $I$ such that
		\[
			\Delta^s_1 \hat{\phi}(0)< 0.
		\]
		
		As $u(x+tz)\in L_{s}(\mathbb{R}),$ by \cite[Lemma 3.8]{kkl}, we have that 
		$\Delta^s_1 \hat{\phi}(t)$ is a continuous function. Therefore, there is
		$\delta>0$ such that 
		\[
			\Delta^s_1 \hat{\phi}(t)< 0 \text{ in } (-\delta,\delta)\subset I	.
		\]
		
		We now take, $x_0=x+\delta z,$  $y_0=x-\delta z.$  If $v$ is 
		the viscosity solution of
		\[
		\left\{
			\begin{array}{ll}
				\Delta^s_1 v(tx_0+(1-t)y_0)=0 &\text{in } t\in (0,1),\\[5pt]
				 v(tx_0+(1-t)y_0)=u(tx_0+(1-t)y_0) &\text{in } t \in\mathbb{R}\setminus (0,1),
			\end{array}
			\right.
		\]
		then $v(tx_0+(1-t)y_0)\in C^\infty(0,1)$ (see, for instance \cite{ross}) and
		\[
			u(tx_0+(1-t)y_0)\le v(tx_0+(1-t)y_0) \text{ in } (0,1)
		\]
		due to the fact that $u$ is $s-$convex. Therefore,
		\begin{equation}\label{eq:wtilde}
			\tilde{w}(0)=\inf_{r>0}\sup \Big\{u(tx_0+(1-t)y_0)\colon t\in(-r,r)\Big\}\le v(x).
		\end{equation} 
		
		On the other hand,
		\[
			z(t)=v\left(\frac{t+\delta}{2\delta}x_0 +\left(1-\frac{t+\delta}{2\delta}\right)y_0\right)
		\] 
		is a viscosity solution to
		\[
		\left\{
			\begin{array}{ll}
				\Delta^s_1 z(t)=0 &\text{in } t\in (-\delta,\delta),\\[6pt]
				 z(t)=u(x+tz) &\text{in } t \in\mathbb{R}\setminus (0,1).
			\end{array}
			\right.
		\]
		Thus, by the strong maximum principle, 
		(that holds just by evaluating the operator in a maximum (or minimum) point),
		we have that
		\[
			 z(t)< \hat{\phi }(t)\quad \forall t\in(-\delta,\delta).
		\]
		Therefore,  using \eqref{eq:wtilde}, we have that
		\[
			\tilde{w}(0)\le v(x)=z(0)<\hat{\phi }(0)=\tilde{w}(0)
		\]
		a contradiction. This finishes the proof. 
	\end{proof}
		
		Notice that we have proved that $u\colon\mathbb{R}^N\to\mathbb{R}$ 
		is $s-$convex in $\Omega$ {\it if and only if}
		$u$ is a viscosity subsolution to 
		\begin{equation}\label{eq:Lu.55.99}
			\Lambda^s_1 u(x)= 0 \quad \text{ in } \Omega
		\end{equation} 
		in the sense of Definition \ref{defi.2.5}. This fact will be stated as Lemma \ref{caract.usc} in Section \ref{sectTeo11}.
	
\section{Classical convexity vs $s-$convexity}\label{sec:convexvssconvex}
	In this section we study when a convex function turns out to be a $s-$convex function. 
	
	\begin{proposition}\label{pop:convex_implies_sconvex}
		Let $s>\tfrac12.$
		If $u$ is a convex function in $\mathbb{R}^N$  such that
		$t \mapsto u(x+tz)\in L_{s}(\mathbb{R})$ for any $x\in \mathbb{R}^N$ 
		and any $z\in\mathbb{S}^{N-1},$ 
		then $u$ is $s-$convex in $\mathbb{R}^N.$ 
	\end{proposition}
	\begin{proof}
		Fix $x,y\in \mathbb{R}^N,$ and $t\in(0,1),$ 
		since $u$ is a convex function in $\mathbb{R}^N,$
		there is $\xi\in\mathbb{R}^N$ such that
		\begin{equation}\label{eq:aux1}
				\xi \cdot (x-y) (t-t_0)+u(t_0x+(1-t_0)y)
			\le u (tx+(1-t)y)
		\end{equation}
		for all $t\in \mathbb{R}.$
		
		Fix $t_0\in(0,1)$, we take 
		\[
			w(tx+(1-t)y)\coloneqq\xi \cdot (x-y) (t-t_0)+u(t_0x+(1-t_0)y)
		\]
		for any $t\in\mathbb{R}.$ Since $w$ is a  function that is affine with respect to $t$ and $s>\tfrac{1}{2}$ we have
		that
		\begin{equation}\label{eq:aux2}
			\Delta^s_1 w (tx+(1-t)y)=0 \quad \forall t\in\mathbb{R}.
		\end{equation}
		
		Thus, by \eqref{eq:aux1} and \eqref{eq:aux2}, we have
		\begin{equation}\label{eq:sconvex2}
		\left\{
			\begin{array}{ll}
				\Delta^s_1 {w}(tx+(1-t)y)= 0 &\text{for all } t\in (0,1);\\[5pt]
				 w(z)\le u(z) &\text{for all } z=tx+(1-t)y  \text{ with }t\not\in (0,1).
			\end{array}
			\right.
		\end{equation}
		Therefore, 
		if $v$ is the viscosity solution of \eqref{eq:sconvex},  
		then by comparison
		\[
			w(tx+(1-t)y)\le v(tx+(1-t)y) \quad\forall t\in(0,1).
		\] 
		In particular 
		\[
			 v(t_0x+(1-t_0)y)\ge w(t_0x+(1-t_0)y)=u(t_0x+(1-t_0)y).
		\]
		
		As $t_0$ is arbitrary, we conclude $u$ is $s-$convex.
	\end{proof}
	
	In the next example, we show that a classical convex function in an interval is
	 not necessarily $s-$convex. This holds because being $s-$convex depends on the values of the function in the whole space, while for being convex only the values inside the domain matter.  
	
	\begin{example}
		Let $u\colon\mathbb{R}\to\mathbb{R}$ be given by
		
		\begin{minipage}{0.4\textwidth}
				    \[
					     u(t)\coloneqq
					            \begin{cases}
					                (t-3)^2 & \text{if } t\in[2,4],\\[6pt]
					                1 & \text{if } |t-3|>1.
					            \end{cases}
					\]
			\end{minipage}
			\begin{minipage}{0.5\textwidth}
				\begin{center}
					    \begin{tikzpicture}[scale=.5]
				            \draw[<->,scale=0.7] (-3,0) -- (6,0) node[right] {$t$};
				            \draw[->,scale=0.7] (0,0) -- (0,3) node[left] {$y$};
						    \draw[thick, blue,rounded corners=1mm,scale=0.7] (-3,1)--(2,1);
					        \draw[thick, blue,rounded corners=1mm,scale=0.7] (4,1)--(6,1);
					        \draw (2,1.3) node[left] {$u(t)$};
					        \draw[domain=2:4,smooth,variable=\x,blue,scale=0.7] plot ({\x},{(\x-3)*(\x-3)});
				        \end{tikzpicture}
			  	\end{center}	
		\end{minipage}
		
		Observe that $u$ is a convex function in $[-1,1]$ (it holds that $u\equiv 1$ in $[-1,1]$).
		
		On the other hand, for any $x,y\in[-1,1],$  
		if $v$ is the viscosity solution of \eqref{eq:sconvex},  
		by the strong maximum principle we have that
		\[
			1\ge v(tx+(1-t)y) \quad\forall t\in(0,1).
		\] 
		Therefore $u$ is not $s-$convex.
	\end{example}
	
	Finally, we also have that the converse does not hold. We present an $s-$convex function in an interval 
	that is not convex.

	\begin{example}
		Let $u\colon\mathbb{R}\to\mathbb{R}$ be the solution to
				\begin{equation}\label{eq:sconvex2.99}
		\left\{
			\begin{array}{ll}
				\Delta^s_1 {u}(z)= 0 &\text{for all } z\in (0,1);\\[5pt]
				 u(z)= g(z) &\text{for all } z \in \mathbb{R}\setminus (0,1)
			\end{array}
			\right.
		\end{equation}
		with $g$  bounded smooth with 
		$g(0)=g(1)=1$, $g\geq 1$ with at least one $x_0\in \mathbb{R}\setminus(0,1)$ 
		such that $g(x_0)>1$.
			
		Observe that $u$ is a $s-$convex function in $(0,1)$. 
		This follows since $u$ solves $\Delta^s_1 {u}= 0$ in $(0,1)$. 
		In fact, take $v$ the solution
		to $\Delta^s_1 {v}= 0$ inside an interval $[x,y] \subset (0,1)$ with exterior 
		Dirichlet datum $u$, as $u$ 
		solves $\Delta^s_1 {u}= 0$ in the interval $[x,y]$ with $u$ as exterior datum, 
		by uniqueness of solutions
		to the $1-$d fractional laplacian, we have $v=u$ in $[x,y]$ (in particular, 
		$u\leq v$ in $[x,y]$) showing that $u$ is $s-$convex.
		
		On the other hand, $u$ is smooth, continuous up to the boundary and, by the 
		strong maximum principle, it holds that 
		\[
			1< u(x) \quad\forall x\in(0,1)
		\] 
		together with $u(0)=u(1)=1$.
		Therefore, $u$ is not convex in $(0,1)$.
	\end{example}

\section{$s-$convex envelope} \label{sectTeo11}

\subsection{ Definition of the $s-$convex envelope }

	Let us call $H(g)$ the set of $s-$convex functions that are below $g$ outside $\Omega$,
	\[
		H(g) \coloneqq 
		\Big\{u \colon u \mbox{ is $s-$convex in $\overline{\Omega}$ and verifies } u|_{\mathbb{R}^N\setminus \Omega} \leq g \Big\}.
	\]
	
	As a consequence of our previous results we have that functions in $H$ are subsolutions to our problem 	
	\eqref{convex-envelope-s-eq} in the sense of Definition \ref{defi.2.5}.
	In fact, as a direct consequence of Theorems \ref{theorem:Luconvex1} and \ref{theorem:Luconvex2}, 
	we have the following lemma.

	\begin{lemma} \label{caract.usc} 
		Let $u\in\mathcal{L}_s(\mathbb{R}^N).$
		Then, $u\in H(g)$ if only if $u$ is a viscosity solution to
		\begin{equation} \label{convex-envelope-eq}
			\begin{array}{ll}
				\Lambda_1^s u (x) \geq 0 \qquad & x \in \Omega, \\[6pt]
				u (x) \leq g (x) \qquad & x \in \mathbb{R}^N\setminus \Omega,
			\end{array}
		\end{equation}
		in the sense of Definition \ref{defi.2.5} (testing with $1-$dimensional functions on segments).
	\end{lemma}	
	
	Recall that the $s-$convex envelope of an exterior 
	datum $g \colon \mathbb{R}^N\setminus \Omega \to  {\mathbb{R}}$ is given by
	\begin{equation} \label{convex-envelope-s.55}
		u^*(x)\coloneqq 
		\sup \Big\{w(x):w\in H(g)\Big\}.
		\end{equation}
		Notice that, from the definition of being $s-$convex, we have that
		$$
		t\to w(x+tz)\in L_{s}(\mathbb{R}), \quad \forall x\in \mathbb{R}^N,
		\forall z\in \mathbb{S}^{N-1} .
		$$
		
	Now we are ready to proceed with the proof of Theorem \ref{teo.1.intro}.

	\begin{proof}[Proof of Theorem \ref{teo.1.intro}]
			We split the proof in two steps.
			
			\noindent {\it Step 1}. First we show that $u^*(x)\in H(g).$
			To this end we only need to prove that $u^*$ is $s-$convex in $\Omega$ 
			because we have $
			w|_{\mathbb{R}^N\setminus \Omega} \leq g $ for every $w \in H(g)$. Since we have that 
			$s-$convexity is equivalent to be a viscosity subsolution to $\Lambda_1^s (u) =0$ in $\Omega$ this follows 
			easily from the fact that the supremum of subsolutions is also a subsolution. Below we include
			a proof for completeness. 
			
			Fix $x,y\in\Omega.$ For any $u\in H(g)$ we have that if $v_u$ is	
			viscosity solution of \eqref{eq:sconvex}
			then $u(tx+(1-t)y)\le v_u(tx+ty)$ in $(0,1).$
			Now, if $v$ is a viscosity solution to
			\[
			\left\{
				\begin{array}{ll}
					\Delta_1^s v(tx+(1-t)y)=0 & t\in (0,1);\\[6pt]
					 v(tx+(1-t)y)=u^*(tx+(1-t)y) & t \in\mathbb{R}\setminus (0,1),
				\end{array}
				\right.
			\]
			by the maximum principle, we get,  
			$$u(tx+(1-t)y)\le v_u(tx + (1-t)y)\le v(tx + (1-t)y) \text{ in } (0,1),$$ that is,
			\[
				u(tx+(1-t)y)\le v(tx + (1-t)y) \text{ in } (0,1),
			\]
			for any $u\in H(g).$ Thus, taking supremum, we obtain
			\[
				u^*(tx+(1-t)y)\le v(tx + (1-t)y) \text{ in } (0,1).
			\]
			Finally, as $x,y\in\Omega$ are arbitrary, we have that $u^*$ is $s-$convex and
			therefore $u^*\in H(g)$			 
			
			\noindent {\it Step 2}. Now we show that the $s-$convex envelope, $u^*$, is a viscosity solution 
			(in the sense of Definition \ref{defi.2.5})
			of 
			\begin{equation} \label{convex-envelope-s-eq.99} \left\{
			\begin{array}{ll}
				\displaystyle \Lambda_1^s u (x) = 0 \quad & x \in \Omega \\[5pt]
				u (x) = g (x) \quad & x \in \mathbb{R}^N\setminus \Omega.
			\end{array}
			\right.
		\end{equation} 
			
			Since $u^*\in H(g),$ $u^*$ is $s-$convex in $\Omega$ and by Lemma \ref{caract.usc}, we know that $u^*$ is a subsolution of 
			\eqref{convex-envelope-s-eq.99} in the sense of Definition \ref{defi.2.5}. 
			
			On the other hand, to prove that $u^*$ is a supersolution of 
			\eqref{convex-envelope-s-eq.99} in the sense of Definition \ref{defi.2.5} we argue by contradiction. 
			
			Assume that $u^*$ is lower semicontinuous (otherwise replace $u^*$ by its lower semicontinuous envelope
			in what follows).
			If there is a smooth $N-$dimensional test function $\phi$ that touches $u^*$ from below at $x_0$
			(we can assume that $u^*-\phi$ has a strict minimum at $x_0$ with $u^*(x_0)=\phi (x_0)$) such that 
			$$
			\Lambda_1^s \phi (x_0) = c > 0,
			$$
			then we have that for every $x$ close to $x_0$
			it holds that 
			$$
			\Lambda_1^s \phi (x) \geq \frac{c}{2} > 0,
			$$
			and hence, for every $z$ such that $|z|=1$ we have that
			$$
			\Delta_1^s \phi (x+tz) |_{t=0} \geq \frac{c}{2} > 0
			$$
			(notice that this strict inequality holds for every $z$ since $\Lambda_1^s$ involves an infimum).
			
			Now
			we can modify our function $u^*$ on a small neighborhood of $x_0$, 
			taking
			$$
			\widehat{u} (x) = \max \{ u^* (x) ; \phi (x) + \delta \}.
			$$
			Notice that we have 
			$$
			\widehat{u} (x) \geq \phi (x) + \delta.
			$$
			Since $u^*-\phi$ has a strict minimum at $x_0$ with $u^*(x_0)=\phi (x_0)$ we have that 
			for $\delta$ small there is 
			a small neighborhood of $x$,
			$U$, such that 
			$$
			\widehat{u} (x) =
			\left\{
			\begin{array}{ll}
			\phi (x) + \delta \quad  & x\in U, \\[6pt]
			u^* (x) \quad & x\not\in U.
			\end{array}
			\right.
			$$
			
			Let us check that this function $\widehat{u} $ is $s-$convex. To this end, using 
			Lemma \ref{caract.usc} it is enough if we prove that  $\widehat{u} $ is a subsolution to 
			$ \Lambda_1^s u (x) = 0$ in the sense of Definition \ref{defi.2.5}. 
			Then, take a smooth $1-$dimensional test function $\psi$ that touches $\widehat{u} $
			from above
			at a point $x\in \Omega$ on some line $x+tz$, that is, we have
			$$
			\psi(t) \geq \widehat{u} (x+tz), \qquad \psi (0) = \widehat{u}(x).
			$$
			
			Now, if $x\in U$ we have
			$$
			\Delta_1^s \psi (0) \geq \Delta_1^s  \widehat{u} (x+tz) |_{t=0} \geq
			\Delta_1^s [\phi (x+tz) + \delta]|_{t=0} >0.
			$$
			On the other hand, if $x\not\in U$, we have that 
			$\widehat{u} (x) =u^*(x)$, $\widehat{u} \geq u^*$ and $\psi$ touches $\widehat{u}$ from above
			at $x$ on the line $x+tz$. Hence, we get that $\psi$ touches $\widehat{u}$ from above
			at $x$ on the line $x+tz$. Then, by {\it Step 1}, using that $u^*$ is $s-$convex
			we get 
			$$
			\Delta_1^s \psi (0) \geq 0,
			$$
			and this proves that $\widehat{u} $ is $s-$convex in $\Omega$.
			
			Since $ \widehat{u} $ coincides with $u^*$ outside $U$ and $u^* \leq g$ in $\mathbb{R}^N \setminus \Omega$
			we conclude that $\widehat{u} \in H(g)$. 
			
			As we have 
			$$
			u^*(x_0) = \phi(x_0) < \phi (x_0) + \delta = \widehat{u} (x_0) \leq \sup \Big\{w(x_0):w\in H(g)\Big\} = u^*(x_0)
			$$
			we arrive to a contradiction that proves that $u^*$
			is a supersolution to \eqref{convex-envelope-s-eq.99}.
			
			We have proved that the $s-$convex envelope $u^*$ is 
			a viscosity solution to \eqref{convex-envelope-s-eq.99} in the sense of Definition \ref{defi.2.5} (testing 
			with $1-$dimensional functions on segments).
			
			 Finally, we observe that every solution to 	
			 \eqref{convex-envelope-s-eq.99} in the sense of Definition \ref{defi.2.5}
			 is also a solution in the sense of Definition \ref{defsol.44}. From the uniqueness result for  
			 	\eqref{convex-envelope-s-eq.99}  proved in Section \ref{sect.2} (working with $N$-dimensional test functions)
				we conclude that the $s-$convex envelope $u^*$ is characterized as the unique viscosity solution to 
					\eqref{convex-envelope-s-eq.99} in the sense of Definition \ref{defsol.44}.
	\end{proof}
	
%
%
	Finally, let us present an example that shows that the strict convexity of the domain is needed
	to have continuity up to the boundary for any continuous data $g$ for the $s-$convex envelope inside $\Omega$.
	
	\begin{example}
		First, we recall that the fact that $\Omega$ is strictly convex is equivalent to the following property:
		Given $y\in\partial\Omega$ we have that for every
		$r>0$ there exists $\delta>0$ such that for every $x\in B(y,\delta)\cap \Omega$ and 
		every direction $z$ ($|z|=1$) 
		it holds that
		\begin{equation}
			\label{condG1} 
			\{x+tz\}_{t\in \R}\cap B(y,r)\cap \partial\Omega\neq\emptyset.
		\end{equation}
		See \cite{BlancRossi}.

		Therefore, when $\Omega$ is not strictly convex there exist a point 
		$y\in\partial\Omega,$ a radius
		 $r>0$ and a sequence of points $x_n \in \Omega$, $x_n \to y$, 
		 and directions $z_n$ ($|z_n|=1$) 
		 such that
		\begin{equation}
		\label{condG1.nop} 
		\{x_n+tz_n\}_{t\in \R}\cap B(y,r)\cap (\R^N\setminus\Omega)= \emptyset.
		\end{equation}

		Now, consider a nonnegative continuous datum $g$ in $\mathbb{R}^N \setminus \Omega$ such that
		$$ 
		\begin{array}{ll}
		g (x)\equiv 0 \qquad & x \in (\mathbb{R}^N \setminus \Omega) \setminus B(y,r), \\[6pt]
		g(y) =1.
		\end{array}
		$$
		Notice that such a function $g$ is necessarily bounded.

		For this $g$ take $u^*$ the $s-$convex envelope inside $\Omega$ (that is well defined since $w\equiv 0$ is
		$s-$convex and verifies $w\leq g$ in $ \mathbb{R}^N \setminus \Omega$).

		Now, at any point $x_n$ in the sequence we consider the direction $z_n$ such that \eqref{condG1.nop}
		holds. Then for any $u\in H(g)$ we have 
		$$
			u(x_n) \leq v(x_n) 
		$$
		with $v$ the solution to the $1-$dimensional fractional laplacian in the 
		line with direction $z_n$ an exterior 
		datum $g$. Then 
		$$
			u^*(x_n) \leq v(x_n) 
		$$
		As $g=0$ in $(\mathbb{R}^N \setminus \Omega) \setminus B(y,r)$ and the line
		$\{x_n+tz_n\}_{t\in \R}$ does not intersects $\mathbb{R}^N \setminus \Omega$ inside the ball
		$B(y,r)$ we get that the exterior condition for $v$ is identically zero, and hence $v(x_n)=0$.
		We conclude that $x_n \in \Omega \to y \in \partial \Omega$ with
		$$
		\limsup_n u^*(x_n) \leq 0 < 1= g(y).
		$$
		This shows that in this case the datum $g$ is not attained continuously. 
\end{example}

\begin{remark} \label{remarklll}
Notice that the definition of the $s-$convex envelope $u^*$ of an exterior continuous and bounded datum $g$ makes sense
	for every domain $\Omega$ (strictly convex or not) and that our previous arguments show that $u^*$ is the 
	largest viscosity solution to 
	$$
	\Lambda_1^s (u) = 0 \qquad \mbox{ in } \Omega
	$$
	with $u\leq g$ in $\mathbb{R}^N \setminus \Omega$.
\end{remark}

\section{Localized $s-$convexity} \label{sect-localiz}
	One can localize $s-$convexity in $\Omega$ 
	and use only values of $u$ inside the domain.  We say that 
	$u\in\mathcal{L}_s(\mathbb{R}^N)$  is {\it locally $s-$convex} in $\Omega$ if for every pair of 
points $x$, $y$ in $\Omega$ such that the segment that joins $x$ and $y$, $[x,y]$ is inside $\Omega$,  then, it holds that
\begin{equation} \label{convexo-s.99}
		u(t x+(1-t)y) \leq v(t x+(1-t)y), \qquad \forall t \in (0,1)
	\end{equation}
	where now $v$ is the viscosity solution 
	to 
	\[
		\Delta^s_{1,\Omega} v(t x+(1-t )y)
		\coloneqq \int_{\{ r: rx+(1-r)y \in \Omega \}} 
		\frac{v(rx+(1-r)y)-v(t x+(1-t)y)}{|r-t|^{1+2s}} \, dr = 0
	\]
	for every $t \in (0,1)$ with 
	\[
		v(z)=u(z) \qquad \mbox{ for } z=tx+(1-t)y \in \Omega \mbox{ with }t\not\in (0,1).
	\] 
Notice that we are restricting the integrals to the part of the line $L_{z}(x)$ with $z=\frac{y-x}{\|y-x\|}$ that is inside $\Omega$ and therefore we are using only values
of $u$ inside $\Omega$ to decide whether $u$ is $s-$convex. 
This strategy to localize inside $\Omega$ is similar to the one that one follows to 
define the restricted fractional laplacian,
$$
\Delta^s_{\Omega} v(x)
		\coloneqq \int_{ \Omega } \frac{v(y)-v(x)}{|x-y|^{N+2s}} \, dy.
$$

With this localized definition, provided that $s>1/2$ (to have a well-defined trace on $\partial \Omega$ along lines),
one can look at the $s-$convex envelope of a boundary datum $g$
defined on $\partial \Omega$. The equation for this $s-$convex envelope is like the previous one but now 
we take the infimum of $1-$dimensional $s-$laplacians integrating in the line intersected with the set $\Omega$,
$$
\displaystyle \Lambda_{1,\Omega}^{s} u (x) \coloneqq \inf 
				\left\{ \int_{L_z(x)\cap \Omega} 
					\frac{u(x+tz)-u(x)}{|t|^{1+2s}} 
					\, dt\colon z\in \mathbb{S}^{N-1}
				\right\} = 0.
$$

For $s>\frac{1}2$,
we can also define $s-$convexity asking
\begin{equation} \label{convexo-s.77}
		u(t x+(1-t)y) \leq v(t x+(1-t)y), \qquad \forall t \in (0,1)
	\end{equation}
	where now $v$ is the viscosity solution 
	to 
	\[
		\Delta^s_{1,[x,y]} v(t x+(1-t )y)
		\coloneqq \int_{s\in  (0,1)} \frac{v(sx+(1-s)y)-v(t x+(1-t)y)}{|s-t|^{1+2s}} \, ds = 0
	\]
	for every $t \in (0,1)$ with 
	\[
		v(x)=u(x), \qquad \mbox{and} \qquad x(y)=u(y).
	\] 
This definition of $s-$convexity only uses the values of $u$ at the endpoints of the segment
(remark that for $s>\frac1{2}$ we have a trace at the boundary of every segment inside $\Omega$). 
	
	With this slightly different localized definition one can also look at the $s-$convex envelope of a boundary datum $g$
defined on $\partial \Omega$. The equation for this $s-$convex envelope is like the previous one but now 
we take the infimum of $1-$dimensional $s-$laplacians integrating in the segment (connected component) of the line 
intersected with the set $\Omega$ that contains $x$,
$$
\displaystyle \widetilde{\Lambda}_{1,\Omega}^{s} u (x) \coloneqq \inf 
				\left\{ \int_{A_{z}(x)} 
					\frac{u(x+tz)-u(x)}{|t|^{1+2s}} 
					\, dt\colon z\in \mathbb{S}^{N-1}
				\right\} = 0,
$$
being
	\[
		A_{z}(x) = \left\{t\colon x+rz \in \Omega, \forall r \in (0,t) \mbox{ or } \forall r\in (t,0) \right\} 
					\subset L_z(x)\cap \Omega .
	\]

When the domain is strictly convex we have that $A_{x,z}$ coincides with $L_z(x)\cap \Omega$ but this is not necessarily the case
for non-convex domains.

\section*{Acknowledgments}
We want to thank the referee for his/her care in reading the manuscript and for several comments and remarks that helped
us to improve the presentation of the results.

We want to warmly thank E. Topp for several interesting discussions. 

L.D.P. and J.D.R. partially supported by 
CONICET grant PIP GI No 11220150100036CO
(Argentina), PICT-2018-03183 (Argentina) and UBACyT grant 20020160100155BA (Argentina).

A. Q. was partially supported by Fondecyt Grant No. 1190282 and Programa Basal, CMM. U. de Chile

	\bigskip
	
{\bf On behalf of all authors, the corresponding author states that there is no conflict of interest. }

	{\bf No data associate for the submission.}

\end{document}